\documentclass[12pt, reqno]{amsart}
\usepackage{eurosym}
\usepackage{amsmath, amsthm, amscd, amsfonts, amssymb, graphicx, color}
\usepackage[bookmarksnumbered, colorlinks, plainpages]{hyperref}

\hypersetup{colorlinks=true,linkcolor=red, anchorcolor=green, citecolor=cyan, urlcolor=red, filecolor=magenta, pdftoolbar=true}
\newtheorem{theorem}{Theorem}[section]
\newtheorem{lemma}[theorem]{Lemma}
\newtheorem{proposition}[theorem]{Proposition}
\newtheorem{corollary}[theorem]{Corollary}
\theoremstyle{definition}
\newtheorem{definition}[theorem]{Definition}

\theoremstyle{remark}
\newtheorem{remark}[theorem]{Remark}
\numberwithin{equation}{section}

\begin{document}
\title[Positive $m$-homogeneous polynomial ideals]{Positive $m$-homogeneous
polynomial ideals}
\author{ Adel Bounabab and Khalil Saadi}
\date{}
\dedicatory{Laboratory of Functional Analysis and Geometry of Spaces,
Faculty of Mathematics and Computer Science, University of Mohamed
Boudiaf-M'sila, Po Box 166, Ichebilia, 28000, M'sila, Algeria.\\
adel.bounabab@univ-msila.dz\\
khalil.saadi@univ-msila.dz}

\begin{abstract}
We introduce and study the concept of positive polynomial ideals between
Banach lattices. The paper develops the basic principles of these classes
and presents methods for constructing positive polynomial ideals from given
positive operator ideals. In addition, we provide concrete examples of
positive polynomial ideals that illustrate the relevance and significance of
these classes.
\end{abstract}

\maketitle

\setcounter{page}{1}


\let\thefootnote\relax\footnote{
\par
{}\textit{Key words and phrases. }Banach lattice, Positive $p$-summing
operators, Cohen positive strongly $p$-summing multilinear operators,
Positive left polynomial ideal, Positive right polynomial ideal, Positive
ideal, Positive $(q;r)$-dominated.polynomial mappings}

\section{\textsc{Introduction and preliminaries}}

Several attempts have been made to extend the theory of operator ideals to
nonlinear contexts. After the development of multi-ideal theory for
multilinear mappings and homogeneous polynomials, attention naturally turned
to the study of positive classes. In \cite{FerSaaP}, the positive ideal of
linear and multilinear mappings was investigated. These new ideals encompass
several important classes of operators, such as positive $p$-summing
operators \cite{Bla87}, positive strongly $p$-summing operators \cite{AB14},
and positive $(p,q)$-dominated operators \cite{CBD21} in the linear case, as
well as Cohen positive strongly $p$-summing multilinear operators \cite{BB18}%
, positive Cohen $p$-nuclear multilinear operators \cite{BBH21}, and
factorable positive strongly $p$-summing multilinear operators \cite{BBR23}
in the multilinear case. Motivated by these studies, we now propose to
define ideals for positive classes of homogeneous polynomials. Our aim is to
examine certain families of positive polynomial ideals and to demonstrate
how positivity enhances the theory of homogeneous polynomials. Following the
same procedure as in \cite{FerSaaP}, we introduce these polynomial ideals in
a systematic way. This approach not only unifies several notions already
studied in the literature but also opens new directions for extending
classical results to the nonlinear and positive framework. As in the linear
and multilinear cases, the definition of positive polynomial ideals
encompasses several classes, such as Cohen positive strongly $p$-summing
polynomials \cite{hamdi}, positive Cohen $p$-nuclear polynomials \cite%
{Hammou}, and positive $p$-dominated polynomials \cite{hamdi}, as well as
the new class we introduce here, namely positive $(q,r)$-dominated
polynomials.

The paper is structured as follows.

In Section 1, we review the fundamental concepts and terminology used in
this work, including Banach lattices, linear operators, symmetric
multilinear forms, and polynomials. We also recall the definition of
positive operator ideals, with particular attention to positive $p$-summing
operators, as well as the definition of polynomial ideals. In Section 2, we
establish the foundations of positive left ideals, denoted $\mathcal{P}%
_{L}^{+}$, positive right ideals, denoted $\mathcal{P}_{R}^{+}$, and
positive ideals, denoted $\mathcal{P}^{+},$ for $m$-homogeneous polynomials.
This framework naturally extends to the positive $m$-linear setting. We then
apply the composition method to construct a positive left polynomial ideal
from a given positive left operator ideal. In addition, we present the
factorization method to generate a positive right ideal of $m$-homogeneous
polynomials from a given positive right operator ideal$.$ In Section 3, we
present concrete examples of positive polynomial ideals, such as Cohen
positive strongly $p$-summing polynomials, positive Cohen $p$-nuclear
polynomials, and positive $p$-dominated polynomials. In particular, we
introduce the class of positive $(q,r)$-dominated polynomials. These classes
satisfy the Pietsch factorization theorem and constitute important examples
of positive polynomial ideals.

Throughout the paper, $E,F$ and $G$ denote Banach lattices and $X,Y$ denote
Banach spaces over $\mathbb{K}$ $\left( \mathbb{R}\text{ or }\mathbb{C}%
\right) $. By $B_{X}$ we denote the closed unit ball of $X$ and by $X^{\ast
} $ its topological dual. We use the symbol $\mathcal{L}(X;Y)$ for the space
of all bounded linear operators from $X$ into $Y$.\ For $1\leq p\leq \infty $%
, we denote by $p^{\ast }$ its conjugate, i.e., $1/p+1/p^{\ast }=1$. Let $E$
be a Banach lattice with norm $\left\Vert \cdot \right\Vert $ and order $%
\leq $. We denote by $E^{+}$ the positive cone of $E$, i.e., $E^{+}=\{x\in
E:x\geqslant 0\}.$ Let $x\in E,$ its positive part is defined by $%
x^{+}:=\sup \{x,0\}\geq 0\ $and its negative part is defined by $x^{-}:=\sup
\{-x,0\}\geq 0.$We have $x=x^{+}-x^{-}$ and $\left\vert x\right\vert
=x^{+}+x^{-}.$ The dual $E^{\ast }$ of a Banach lattice $E$ is a Banach
lattice with the natural order 
\begin{equation*}
x_{1}^{\ast }\leq x_{2}^{\ast }\Leftrightarrow \langle x,x_{1}^{\ast
}\rangle \leq \langle x,x_{2}^{\ast }\rangle ,\forall x\in E^{+}.
\end{equation*}%
A bounded linear operator $u:E\rightarrow F$ is called positive if $u(x)\in
F^{+}$, whenever $x\in E^{+}$. By $\mathcal{L}^{+}(E;F)$ we denote the set
of all positive operators from $E$ to $F$. A linear operator $u$ is called 
\textit{regular} if there exist $u_{1},u_{2}\in \mathcal{L}^{+}(E;F)$ such
that%
\begin{equation*}
u=u_{1}-u_{2}.
\end{equation*}%
We denote by $\mathcal{L}^{r}(E;F)$ the vector space of regular operators
from $E$ to $F.$ The vector space $\mathcal{L}^{r}(E;F)$ is generated by
positive operators which is a Banach space with the norm%
\begin{equation*}
\left\Vert u\right\Vert _{r}=\inf \left\{ \left\Vert v\right\Vert :v\in 
\mathcal{L}^{+}(E;F),\left\vert u\left( x\right) \right\vert \leq v\left(
x\right) ,x\in E^{+}\right\} ,
\end{equation*}%
By \cite[Section 1.3]{MN91}, if $F=\mathbb{K}$, we have 
\begin{equation*}
E^{\ast }=\mathcal{L}(E,\mathbb{K})=\mathcal{L}^{r}(E,\mathbb{K}).
\end{equation*}%
The canonical embedding $i:E\longrightarrow E^{\ast \ast }$ such that $%
\left\langle i(x),x^{\ast }\right\rangle =\left\langle x^{\ast
},x\right\rangle $ of $E$ into its second dual $E^{\ast \ast }$ is an order
isometry from $E$ onto a sublattice of $E^{\ast \ast }$. If we consider $E$
as a sublattice of $E^{\ast \ast }$ we have for $x_{1},x_{2}\in E$ 
\begin{equation*}
x_{1}\leq x_{2}\Longleftrightarrow \left\langle x_{1},x^{\ast }\right\rangle
\leq \left\langle x_{2},x^{\ast }\right\rangle ,\quad \forall x^{\ast }\in
E^{\ast +}.
\end{equation*}%
The spaces $\mathcal{C}(K)$ where $K$ compact and $L_{p}(\mu )$, $(1\leq
p\leq \infty )$ are Banach lattices. Let $X$ be a Banach space. We denote by 
$\ell _{p}^{n}(X)$ the Banach space of all absolutely $p$-summable sequences 
$(x_{i})_{i=1}^{n}\subset X$ with the norm 
\begin{equation*}
\Vert (x_{i})_{i=1}^{n}\Vert _{p}=(\sum_{i=1}^{n}\Vert x_{i}\Vert ^{p})^{%
\frac{1}{p}},
\end{equation*}%
and by $\ell _{p,w}^{n}(X)$ the Banach space of all weakly $p$-summable
sequences $(x_{i})_{i=1}^{n}\subset X$ with the norm,%
\begin{equation*}
\Vert (x_{i})_{i=1}^{n}\Vert _{p,w}=\sup_{x^{\ast }\in B_{X^{\ast
}}}(\sum_{i=1}^{n}|\langle x^{\ast },x_{i}\rangle |^{p})^{\frac{1}{p}}.
\end{equation*}%
Consider the case where $X$ is replaced by a Banach lattice $E$, and define 
\begin{equation*}
\ell _{p,|w|}^{n}(E)=\{(x_{i})_{i=1}^{n}\subset E:\left( |x_{i}|\right)
_{i=1}^{n}\in \ell _{p,w}^{n}(E)\}\text{ and }\Vert (x_{i})_{i=1}^{n}\Vert
_{p,|w|}=\Vert (|x_{i}|)_{i=1}^{n}\Vert _{p,w}.
\end{equation*}%
Let $B_{E^{\ast }}^{+}=\{x^{\ast }\in B_{E^{\ast }}:x^{\ast }\geq
0\}=B_{E^{\ast }}\cap E^{\ast +}$. If $(x_{i})_{i=1}^{n}\subset E^{+}$ , we
have that 
\begin{equation*}
\Vert (x_{i})_{i=1}^{n}\Vert _{p,|w|}=\Vert (x_{i})_{i=1}^{n}\Vert
_{p,w}=\sup_{x^{\ast }\in B_{E^{\ast }}^{+}}(\sum_{i=1}^{n}\langle x^{\ast
},x_{i}\rangle ^{p})^{\frac{1}{p}}.
\end{equation*}%
Given $m\in \mathbb{N},$ we denote by $\mathcal{L}(X_{1},...,X_{m};Y)$ the
Banach space of all bounded multilinear operators from $X_{1}\times
...\times X_{n}$ into $Y$ endowed with the supremum norm%
\begin{equation*}
\left\Vert T\right\Vert =\sup_{\substack{ \left\Vert x_{i}\right\Vert \leq 1 
\\ \left( 1\leq i\leq m\right) }}\left\Vert T\left( x_{1},...,x_{m}\right)
\right\Vert .
\end{equation*}%
A map $P:X\rightarrow Y$ is an $m$-homogeneous polynomial if there exists a
unique symmetric $m$-linear operator $\widehat{P}:X\times ...\times
X\rightarrow Y$ such that%
\begin{equation*}
P\left( x\right) =\widehat{P}\left( x,\overset{(m)}{...},x\right) ,
\end{equation*}%
for every $x\in X.$ Both are related by the polarization formula \cite[%
Theorem 1.10]{Muji} 
\begin{equation*}
\widehat{P}\left( x_{1},...,x_{m}\right) =\frac{1}{m!2^{m}}\sum\limits 
_{\substack{ \epsilon _{i}=\pm  \\ 1\leq i\leq m}}\epsilon _{1}...\epsilon
_{m}P(\sum\limits_{j=1}^{m}\epsilon _{j}x_{j}).
\end{equation*}%
We denote by $\mathcal{P}\left( ^{m}X;Y\right) $, the Banach space of all
continuous $m$-homogeneous polynomials from $X$ into $Y$ endowed with the
norm%
\begin{equation*}
\left\Vert P\right\Vert =\sup_{\left\Vert x\right\Vert \leq 1}\left\Vert
P\left( x\right) \right\Vert =\inf \left\{ C:\left\Vert P\left( x\right)
\right\Vert \leq C\left\Vert x\right\Vert ^{m},x\in X\right\} .
\end{equation*}%
We denote by $\mathcal{P}_{f}(^{m}X;Y)$ the space of all $m$-homogeneous
polynomials of finite type, that is%
\begin{equation*}
\mathcal{P}_{f}(^{m}X;Y)=\left\{ \sum\limits_{i=1}^{k}\varphi _{i}^{m}\left(
x\right) y_{i}:k\in \mathbb{N},\varphi _{i}\in X^{\ast },y_{i}\in Y,1\leq
i\leq k\right\} .
\end{equation*}%
If $Y=\mathbb{K}$, we write simply $\mathcal{P}\left( ^{m}X\right) $. For
the general theory of polynomials on Banach spaces, we refer to \cite{Dini}
and \cite{Muji}. By $X_{1}\widehat{\otimes }_{\pi }...\widehat{\otimes }%
_{\pi }X_{m}$\ we denote the completed projective tensor product of $%
X_{1},...,X_{m}.$ If $X=X_{1}=...=X_{m},$ we write $\widehat{\otimes }_{\pi
}^{m}X$. By $\otimes _{s}^{m}X:=X\otimes \overset{(m)}{...}\otimes X$ we
denote the m fold symmetric tensor product of $X$, that is, 
\begin{equation*}
\otimes _{s}^{m}X=\left\{ \sum\limits_{i=1}^{n}\lambda _{i}x_{i}\otimes 
\overset{(m)}{...}\otimes x_{i}:n\in \mathbb{N},\lambda _{i}\in \mathbb{K}%
,x_{i}\in X,\left( 1\leq i\leq n\right) \right\} .
\end{equation*}%
By $\widehat{\otimes }_{\pi ,s}^{m}X$ we denote the closure of $\otimes
_{s}^{m}X$ in $\widehat{\otimes }_{\pi }^{m}X.$ For symmetric tensor
products, we refer to \cite{Flo}. Let $P\in \mathcal{P}\left( ^{m}X;Y\right)
,$ we define its linearization $P_{L}:\widehat{\otimes }_{\pi
,s}^{m}X\rightarrow Y$ by 
\begin{equation*}
P_{L}\left( x\otimes \overset{(m)}{...}\otimes x\right) =P\left( x\right) ,
\end{equation*}%
for all $x\in X.$ Consider the canonical $m$-homogeneous polynomial $\delta
_{m}:X\rightarrow \widehat{\otimes }_{\pi ,s}^{m}X$ defined by 
\begin{equation*}
\delta _{m}\left( x\right) =x\otimes \overset{(m)}{...}\otimes x.
\end{equation*}%
We have the next diagram which is commute%
\begin{equation*}
\begin{array}{ccc}
X & \overset{P}{\rightarrow } & Y \\ 
& \delta _{m}\searrow & \uparrow P_{L} \\ 
&  & \widehat{\otimes }_{\pi ,s}^{m}X%
\end{array}%
\end{equation*}%
in the other words, $P=P_{L}\circ \delta _{m}.$ We have $\left\Vert
P\right\Vert =\left\Vert P_{L}\right\Vert $ and we have the following
isometric identification%
\begin{equation*}
\mathcal{P}(^{m}X;E)=\mathcal{L}(\widehat{\otimes }_{\pi ,s}^{m}X;E).
\end{equation*}

Blasco \cite{Bla87} introduced the positive generalization of $p$-summing
operators as follows: An operator $u:E\longrightarrow X$ is said to be 
\textit{positive }$p$\textit{-summing} ($1\leq p<\infty $) if there exists a
constant $C>0$ such that the inequality 
\begin{equation}
\left\Vert \left( u\left( x_{i}\right) \right) _{i=1}^{n}\right\Vert
_{p}\leq C\left\Vert \left( x_{i}\right) _{i=1}^{n}\right\Vert _{p,w},
\label{sec1def1}
\end{equation}%
holds for all $x_{1},\ldots ,x_{n}\in E^{+}.$ We denote by $\Pi
_{p}^{+}(E;X) $, the space of positive $p$-summing operators from $E$ into $%
X,$ which is a Banach space with the norm $\pi _{p}^{+}(T)$ given by the
infimum of the constant $C>0$ that verify the inequality $(\ref{sec1def1}).$
We have $\Pi _{\infty }^{+}(E;X)=\mathcal{L}(E;X)$. O.I. Zhukova \cite{Zhu98}%
, gives the Pietsch domination theorem concerning this class. The operator $%
u $ belongs to $\Pi _{p}^{+}(E;X)$ if and only if there exist a Radon
probability measure $\mu $ on the set $B_{E^{\ast }}^{+}$ and a positive
constant $C$ such that for every $x\in E^{+}$%
\begin{equation}
\left\Vert u\left( x\right) \right\Vert \leq C\left( \int_{B_{E^{\ast
}}^{+}}\langle x,x^{\ast }\rangle ^{p}d\mu (x^{\ast })\right) ^{\frac{1}{p}}.
\label{DomiSumming}
\end{equation}

\textbf{Positive operator ideal}: We provide the definition of the positive
ideal introduced and studied in \cite{FerSaaP}: A positive left ideal ,
denoted by $\mathcal{B}_{L}^{+},$ is a subclass of all continuous linear
operators from a Banach space into a Banach lattice such that for every
Banach space $X$ and Banach lattice $E,$ the components 
\begin{equation*}
\mathcal{B}_{L}^{+}\left( X;E\right) :=\mathcal{L}\left( X;E\right) \cap 
\mathcal{B}_{L}^{+}
\end{equation*}%
satisfy:\newline
$(i)$ $\mathcal{B}_{L}^{+}\left( X;E\right) $ is a linear subspace of $%
\mathcal{L}\left( X;E\right) $ containing the linear mappings of finite rank.%
\newline
$(ii)$ The positive ideal property: If $T\in \mathcal{B}_{L}^{+}\left(
X;E\right) ,u\in \mathcal{L}\left( Y;X\right) $ and $v\in \mathcal{L}%
^{+}\left( E;F\right) $, then $v\circ T\circ u$ is in $\mathcal{B}%
_{L}^{+}\left( Y;F\right) $.\newline
If $\Vert \cdot \Vert _{\mathcal{B}_{L}^{+}}:\mathcal{B}_{L}^{+}\rightarrow 
\mathbb{R}^{+}$ satisfies:\newline
a) $\left( \mathcal{B}_{L}^{+}\left( X;E\right) ,\Vert \cdot \Vert _{%
\mathcal{B}_{L}^{+}}\right) $ is a Banach space for all Banach space $X$ and
Banach lattice $E$.\newline
b) The form $T:\mathbb{K}\rightarrow \mathbb{K}$ given by $T\left( \lambda
\right) =\lambda $ satisfies $\left\Vert u\right\Vert _{\mathcal{B}%
_{L}^{+}}=1$,\newline
c) $T\in \mathcal{B}_{L}^{+}\left( X;E\right) ,u\in \mathcal{L}\left(
Y;X\right) $ and $v\in \mathcal{L}^{+}\left( E;F\right) $ then 
\begin{equation*}
\left\Vert v\circ T\circ u\right\Vert _{\mathcal{B}_{L}^{+}}\leq \Vert
v\Vert \Vert T\Vert _{\mathcal{B}_{L}^{+}}\left\Vert u\right\Vert .
\end{equation*}%
The class $\left( \mathcal{B}_{L}^{+},\Vert \cdot \Vert _{\mathcal{B}%
_{L}^{+}}\right) $ is a positive Banach ideal.

The \textit{positive right ideal}, denoted $\mathcal{B}_{R}^{+}$, is
obtained by reversing the roles of the operators $u$ and $v$. In this case,
we consider compositions with positive linear operators on the right and
arbitrary linear operators on the left. Similarly, the \textit{positive ideal%
}, denoted $\mathcal{B}^{+}$, is obtained by restricting to positive linear
operators on both sides of the composition. Note that every positive right
or left ideal is automatically a positive ideal.

\textbf{Polynomial ideal}: An ideal of $m$-homogeneous polynomials (or
polynomial ideal) $\mathcal{Q}$ is a subclass of the class of all continuous
homogeneous polynomials between Banach spaces such that for all $m\in 
\mathbb{N}$, and Banach spaces $E$ and $F,$ the components $\mathcal{Q}%
\left( ^{m}E;F\right) =\mathcal{P}\left( ^{m}E;F\right) \cap \mathcal{Q}$
satisfy:\newline
$(i)$ $\mathcal{Q}\left( ^{m}E;F\right) $ is a linear subspace of $\mathcal{P%
}\left( ^{m}E;F\right) $ which contains the $m$-homogeneous polynomials of
finite type.\newline
$(ii)$ The positive ideal property: If $u\in \mathcal{L}\left( G;E\right)
,P\in \mathcal{P}\left( ^{m}E;F\right) $ and $v\in \mathcal{L}\left(
F;G\right) $, then $v\circ P\circ u$ is in $\mathcal{Q}\left( ^{m}G;H\right) 
$.\newline
If $\Vert \cdot \Vert _{\mathcal{Q}}:\mathcal{Q}\rightarrow \mathbb{R}^{+}$
satisfies:\newline
a) $(\mathcal{Q}\left( ^{m}E;F\right) ,\Vert \cdot \Vert _{\mathcal{Q}})$ is
a normed (Banach) space for all Banach spaces $E,F$ and $m$.\newline
b) The polynomial form $P^{m}:\mathbb{K}\rightarrow \mathbb{K}$ given by $%
P\left( \lambda \right) =\lambda ^{m}$ satisfies $\left\Vert
P^{m}\right\Vert _{\mathcal{Q}}=1$,\newline
c) If $u\in \mathcal{L}\left( G;E\right) ,P\in \mathcal{P}\left(
^{m}E;F\right) $ and $v\in \mathcal{L}\left( F;G\right) ,$ then 
\begin{equation*}
\left\Vert v\circ P\circ u\right\Vert _{\mathcal{Q}}\leq \Vert v\Vert \Vert
P\Vert _{\mathcal{Q}}\left\Vert u\right\Vert ^{m}.
\end{equation*}%
The class $(\mathcal{Q},\Vert \cdot \Vert _{\mathcal{Q}})$ is called a
normed (Banach) polynomial ideal. The case $m=1$ recovers the classical
theory of normed and Banach operator ideals. For further details on linear
operator ideals, we refer to \cite{defl}.

\section{\textsc{Positive polynomial ideals}}

In this section, we introduce positive polynomial ideals and develop
abstract methods for their construction in the setting of $m$-homogeneous
polynomials. The positive ideal property is examined through the behavior of
associated linear operators, establishing a natural connection between the
linear and polynomial frameworks. These classes extend and complement the
theory of positive multilinear ideals studied in detail in \cite{FerSaaP}.

\begin{definition}
A positive left polynomial ideal (or positive left ideal of $m$-homogeneous
polynomials), denoted by $\mathcal{P}_{L}^{+},$ is a subclass of all
continuous $m$-homogeneous polynomials from a Banach space into a Banach
lattice such that for all Banach space $X$ and Banach lattice $E,$ the
components 
\begin{equation*}
\mathcal{P}_{L}^{+}\left( ^{m}X;E\right) :=\mathcal{P}\left( ^{m}X;E\right)
\cap \mathcal{P}_{L}^{+},
\end{equation*}%
satisfy:\newline
$(i)$ $\mathcal{P}_{L}^{+}\left( ^{m}X;E\right) $ is a linear subspace of $%
\mathcal{P}\left( ^{m}X;E\right) $ containing the polynomials of finite rank.%
\newline
$(ii)$ The positive ideal property: If $P\in \mathcal{P}_{L}^{+}\left(
^{m}X;E\right) ,u\in \mathcal{L}\left( Y;X\right) $ and $v\in \mathcal{L}%
^{+}\left( E;F\right) $, then $v\circ P\circ u$ is in $\mathcal{P}%
_{L}^{+}\left( ^{m}Y;F\right) $.\newline
If $\Vert \cdot \Vert _{\mathcal{P}_{L}^{+}}:\mathcal{P}_{L}^{+}\rightarrow 
\mathbb{R}^{+}$ satisfies:\newline
a) $\left( \mathcal{P}_{L}^{+}\left( ^{m}X;E\right) ,\Vert \cdot \Vert _{%
\mathcal{P}_{L}^{+}}\right) $ is a Banach (quasi-Banach) space for all
Banach space $X$ and Banach lattice $E$,\newline
b) The polynomial form $P^{m}:\mathbb{K}\rightarrow \mathbb{K}$ given by $%
P^{m}\left( \lambda \right) =\lambda ^{m}$ satisfies $\left\Vert
P^{m}\right\Vert _{\mathcal{P}_{L}^{+}}=1$,\newline
c) If $P\in \mathcal{P}_{L}^{+}\left( ^{m}X;E\right) ,u\in \mathcal{L}\left(
Y;X\right) $ and $v\in \mathcal{L}^{+}\left( E;F\right) ,$ then 
\begin{equation*}
\left\Vert v\circ T\circ u\right\Vert _{\mathcal{P}_{L}^{+}}\leq \Vert
v\Vert \Vert P\Vert _{\mathcal{P}_{L}^{+}}\left\Vert u\right\Vert ^{m}.
\end{equation*}%
The class $\left( \mathcal{P}_{L}^{+},\Vert \cdot \Vert _{\mathcal{P}%
_{L}^{+}}\right) $ is a positive left Banach (quasi-Banach) polynomial ideal.%
\newline
\end{definition}

\begin{remark}
\label{Remark1}In condition $(ii)$, because every regular operator is a
difference of positive ones, the set $\mathcal{L}^{+}\left( E;F\right) $ can
be replaced by the space $\mathcal{L}^{r}\left( E;F\right) $, and condition $%
(ii)$ remains the same.
\end{remark}

Analogous to the previous approach, we introduce the \textit{positive right
polynomial ideal}, denoted $\mathcal{P}_{R}^{+}$, by swapping the roles of
the operators $u$ and $v$. In doing so, we examine the composition of
positive linear operators on the right-hand side and linear operators on the
left-hand side. Similarly, we define the \textit{positive polynomial ideal},
denoted $\mathcal{P}^{+}$, by considering only the positive linear
operators, with composition occurring on both the right and left sides.

\begin{remark}
1) It is evident that every polynomial ideal is indeed positive polynomial
ideal.

2) Every positive right or left polynomial ideal is positive polynomial
ideal.
\end{remark}

\begin{proposition}
\label{Propo}Let $\mathcal{B}_{R}^{+}$ be a positive right operator ideal
and $\mathcal{P}_{L}^{+}$ a positive left polynomial ideal. The composition
ideal $\mathcal{P}_{L}^{+}\circ \mathcal{B}_{R}^{+}$ is defined as the set
of polynomials $P$ that admit a factorization $P=Q\circ u,$ with $u\in 
\mathcal{B}_{R}^{+}\left( E;X\right) $ and $Q\in \mathcal{P}_{L}^{+}\left(
^{m}X;F\right) $. This construction yields a positive polynomial ideal.
\end{proposition}

\begin{proof}
Let $E$ and $F$ be Banach lattices. We will verify that $\mathcal{P}%
_{L}^{+}\circ \mathcal{B}_{R}^{+}\left( ^{m}E;F\right) $ is a linear
subspace. Let $\lambda \in \mathbb{K}$ and $P\in \mathcal{P}_{L}^{+}\circ 
\mathcal{B}_{R}^{+}(^{m}E;F).$ There exist a Banach space $X$ and elements $%
u_{0}\in \mathcal{B}_{R}^{+}\left( E;X\right) ,Q_{0}\in \mathcal{P}%
_{L}^{+}\left( ^{m}X;F\right) $ such that $P=Q_{0}\circ u_{0}.$ Then, $%
\lambda P=\left( \lambda Q_{0}\right) \circ u_{0}\in \mathcal{P}%
_{L}^{+}\circ \mathcal{B}_{R}^{+}(^{m}E;F).$ Now, let $P_{1},P_{2}\in 
\mathcal{P}_{L}^{+}\circ \mathcal{B}_{R}^{+}(^{m}E;F)$ such that there exist
Banach spaces $X,Y$ and elements $u_{1}\in \mathcal{B}_{R}^{+}\left(
E;X\right) ,u_{2}\in \mathcal{B}_{R}^{+}\left( E;Y\right) ,Q_{1}\in \mathcal{%
P}_{L}^{+}\left( ^{m}X;F\right) ,$ and $Q_{2}\in \mathcal{P}_{L}^{+}\left(
^{m}Y;F\right) $ with the following commutative diagrams: 
\begin{equation*}
\begin{array}{ccc}
E & \overset{P_{1}}{\longrightarrow } & F \\ 
u_{1}\downarrow & \nearrow Q_{1} &  \\ 
X &  & 
\end{array}%
\text{ and }%
\begin{array}{ccc}
E & \overset{P_{2}}{\longrightarrow } & F \\ 
u_{2}\downarrow & \nearrow Q_{2} &  \\ 
Y &  & 
\end{array}%
\end{equation*}%
We define $A=i_{1}\circ u_{1}+i_{2}\circ u_{2},$ where $i_{1}:X%
\longrightarrow X\times Y$ and $i_{2}:Y\longrightarrow X\times Y$ are given
by $i_{1}\left( x\right) =\left( x,0\right) $ and $i_{2}\left( y\right)
=\left( 0,y\right) .$ We have%
\begin{equation*}
A\in \mathcal{B}_{R}^{+}\left( E;X\times Y\right) ,
\end{equation*}%
since $u_{1}\in \mathcal{B}_{R}^{+}\left( E;X\right) $ and $u_{2}\in 
\mathcal{B}_{R}^{+}\left( E;Y\right) $ we have $i_{j}\circ u_{j}\in \mathcal{%
B}_{R}^{+}\left( E;X\times Y\right) \left( j=1,2\right) $. Consequently, 
\begin{equation*}
A=i_{1}\circ u_{1}+i_{2}\circ u_{2}\in \mathcal{B}_{R}^{+}\left( E;X\times
Y\right) .
\end{equation*}%
On the other hand, we define $B=Q_{1}\circ \pi _{1}+Q_{2}\circ \pi _{2},$
where $\pi _{1}:X\times Y\longrightarrow X$ and $\pi _{2}:X\times
Y\longrightarrow Y$ are given by $\pi _{1}\left( x,y\right) =x$ and $\pi
_{2}\left( x,y\right) =y.$ We have 
\begin{equation*}
B\in \mathcal{P}_{L}^{+}\left( ^{m}X\times Y;F\right) ,
\end{equation*}%
where $\widehat{B}:\left( X\times Y\right) \times ...\times \left( X\times
Y\right) ,$ the multilinear symmetric associated to $B,$ is defined by%
\begin{equation*}
\widehat{B}=\widehat{Q_{1}}\left( \pi _{1},\overset{\left( m\right) }{...}%
,\pi _{1}\right) +\widehat{Q_{2}}\left( \pi _{2},\overset{\left( m\right) }{%
...},\pi _{2}\right) .
\end{equation*}%
Indeed, let $\left( x,y\right) \in X\times Y,$ we have%
\begin{eqnarray*}
\widehat{B}\left( \left( x,y\right) ,\overset{\left( m\right) }{...},\left(
x,y\right) \right) &=&\widehat{Q_{1}}\left( \pi _{1}\left( x,y\right) ,%
\overset{\left( m\right) }{...},\pi _{1}\left( x,y\right) \right) +\widehat{%
Q_{2}}\left( \pi _{2}\left( x,y\right) ,\overset{\left( m\right) }{...},\pi
_{2}\left( x,y\right) \right) \\
&=&Q_{1}\circ \pi _{1}\left( x,y\right) +Q_{2}\circ \pi _{2}\left(
x,y\right) =B\left( x,y\right) .
\end{eqnarray*}%
Similarly, since $Q_{1}\in \mathcal{P}_{L}^{+}\left( X;F\right) $ and $%
Q_{2}\in \mathcal{P}_{L}^{+}\left( Y;F\right) ,$ we have $Q_{j}\circ \pi
_{j}\in \mathcal{P}_{L}^{+}\left( ^{m}X\times Y;F\right) \left( j=1,2\right) 
$. Consequently, 
\begin{equation*}
B=Q_{1}\circ \pi _{1}+Q_{2}\circ \pi _{2}\in \mathcal{P}_{L}^{+}\left(
^{m}X\times Y;F\right) .
\end{equation*}%
A simple calculation shows that 
\begin{equation*}
P_{1}+P_{2}=B\circ A.
\end{equation*}%
Let $P_{f}\in \mathcal{P}(^{m}E;F)$ be a finite-rank operator. It can be
expressed as a combination of operators of the form $\varphi ^{m}b$ where $%
\varphi \in E^{\ast }$ and $b\in F.$ Let $u=\varphi ^{m}b.$ Define $B:%
\mathbb{K}\longrightarrow F$ by $B\left( \lambda \right) =\lambda
^{m}b=\left( id_{\mathbb{K}}\left( \lambda \right) \right) ^{m}b.$ Clearly, $%
B\in \mathcal{P}_{L}^{+}\left( ^{m}\mathbb{K};F\right) $ and define $%
A:E\longrightarrow \mathbb{K}$ by $A\left( x\right) =\varphi \left( x\right) 
$ which belongs to $\mathcal{B}_{R}^{+}\left( E;\mathbb{K}\right) .$ Then,
we have 
\begin{equation*}
u\left( x\right) =B\circ A\left( x\right) \in \mathcal{P}_{L}^{+}\circ 
\mathcal{B}_{R}^{+}\left( ^{m}E;F\right) .
\end{equation*}%
By the vector space structure of $\mathcal{P}_{L}^{+}\circ \mathcal{B}%
_{R}^{+}\left( ^{m}E;F\right) $ it follows that $P_{f}\in \mathcal{P}%
_{L}^{+}\circ \mathcal{B}_{R}^{+}\left( ^{m}E;F\right) .$ Finally, we verify
the positive ideal property. Let $P=Q_{0}\circ u_{0}\in \mathcal{P}%
_{L}^{+}\circ \mathcal{B}_{R}^{+}\left( ^{m}E;F\right) ,u\in \mathcal{L}%
^{+}\left( G;E\right) $ and $v\in \mathcal{L}^{+}\left( F;H\right) $. Then 
\begin{equation*}
v\circ P\circ u=\left( v\circ Q_{0}\right) \circ \left( u_{0}\circ u\right) .
\end{equation*}%
Since $v\circ Q_{0}\in \mathcal{P}_{L}^{+}\left( ^{m}X;H\right) $ and $%
u_{0}\circ u\in \mathcal{B}_{R}^{+}\left( G;X\right) ,$ we obtain $v\circ
P\circ u\in \mathcal{P}_{L}^{+}\circ \mathcal{B}_{R}^{+}\left(
^{m}E;F\right) .$
\end{proof}

Let $\mathcal{B}_{R}^{+}$ be a positive right Banach ideal and $\mathcal{P}%
_{L}^{+}$ a positive left Banach polynomial ideal. If $E$ and $F$ are Banach
lattices and $P\in \mathcal{P}_{L}^{+}\circ \mathcal{B}_{R}^{+}(^{m}E;F),$
we define%
\begin{equation}
\left\Vert P\right\Vert _{\mathcal{P}_{L}^{+}\circ \mathcal{B}_{R}^{+}}=\inf
\left\{ \left\Vert Q\right\Vert _{\mathcal{P}_{L}^{+}}\left\Vert
u\right\Vert _{\mathcal{B}_{R}^{+}}^{m}:P=Q\circ u\right\} .  \label{2.1}
\end{equation}

\begin{proposition}
\label{Propo1}1) For every $P\in \mathcal{P}_{L}^{+}\circ \mathcal{B}%
_{R}^{+}(^{m}E;F),u\in \mathcal{L}^{+}\left( G;E\right) $ and $v\in \mathcal{%
L}^{+}\left( F;H\right) ,$ we have%
\begin{equation*}
\left\Vert v\circ P\circ u\right\Vert _{\mathcal{P}_{L}^{+}\circ \mathcal{B}%
_{R}^{+}}\leq \Vert v\Vert \left\Vert P\right\Vert _{\mathcal{P}%
_{L}^{+}\circ \mathcal{B}_{R}^{+}}\left\Vert u\right\Vert ^{m}.
\end{equation*}%
2) Let $P^{m}:\mathbb{K}\rightarrow \mathbb{K}$ given by $P^{m}\left(
\lambda \right) =\lambda ^{m}.$ We have $\left\Vert P^{m}\right\Vert _{%
\mathcal{P}_{L}^{+}\circ \mathcal{B}_{R}^{+}}=1.$\newline
3) For every $P\in \mathcal{P}_{L}^{+}\circ \mathcal{B}_{R}^{+}\left(
^{m}E;F\right) ,$ we have 
\begin{equation}
\left\Vert P\right\Vert \leq \left\Vert P\right\Vert _{\mathcal{P}%
_{L}^{+}\circ \mathcal{B}_{R}^{+}}.  \label{2.2}
\end{equation}
\end{proposition}

\begin{proof}
1) Let $P\in \mathcal{P}_{L}^{+}\circ \mathcal{B}_{R}^{+}(^{m}E;F),u\in 
\mathcal{L}^{+}\left( G;E\right) $ and $v\in \mathcal{L}^{+}\left(
F;H\right) .$ Take a representation $P=Q_{0}\circ u_{0}.$ Then 
\begin{eqnarray*}
\left\Vert v\circ P\circ u\right\Vert _{\mathcal{P}_{L}^{+}\circ \mathcal{B}%
_{R}^{+}} &=&\left\Vert \left( v\circ Q_{0}\right) \circ \left( u_{0}\circ
u\right) \right\Vert _{\mathcal{P}_{L}^{+}\circ \mathcal{B}_{R}^{+}} \\
&\leq &\Vert v\circ Q_{0}\Vert _{\mathcal{P}_{L}^{+}}\left\Vert u_{0}\circ
u\right\Vert _{\mathcal{B}_{R}^{+}}^{m} \\
&\leq &\Vert v\Vert \Vert Q_{0}\Vert _{\mathcal{P}_{L}^{+}}\left\Vert
u_{0}\right\Vert _{\mathcal{B}_{R}^{+}}^{m}\left\Vert u\right\Vert ^{m}.
\end{eqnarray*}%
Taking the infimum over all representations of $P$, we obtain%
\begin{equation*}
\left\Vert v\circ P\circ u\right\Vert _{\mathcal{P}_{L}^{+}\circ \mathcal{B}%
_{R}^{+}}\leq \Vert v\Vert \left\Vert P\right\Vert _{\mathcal{P}%
_{L}^{+}\circ \mathcal{B}_{R}^{+}}\left\Vert u\right\Vert ^{m}.
\end{equation*}%
2) Note that $P^{m}=\left( id_{\mathbb{K}}\right) ^{m}\circ id_{\mathbb{K}}.$
Hence%
\begin{equation*}
\left\Vert P^{m}\right\Vert _{\mathcal{P}_{L}^{+}\circ \mathcal{B}%
_{R}^{+}}=\left\Vert \left( id_{\mathbb{K}}\right) ^{m}\circ id_{\mathbb{K}%
}\right\Vert _{\mathcal{P}_{L}^{+}\circ \mathcal{B}_{R}^{+}}\leq \left\Vert
\left( id_{\mathbb{K}}\right) ^{m}\right\Vert _{\mathcal{P}%
_{L}^{+}}\left\Vert id_{\mathbb{K}}\right\Vert _{\mathcal{B}_{R}^{+}}=1.
\end{equation*}%
On the other hand, let $Q_{0}\circ u_{0}$ be a factorization of $P^{m}$ with 
$u_{0}:\mathbb{K}\rightarrow X$ and $Q_{0}:X\rightarrow \mathbb{K}.$ Then,
there exists $x_{0}\in X$ such that 
\begin{equation*}
u_{0}\left( \lambda \right) =\lambda x_{0}\text{ and }Q_{0}\left(
x_{0}\right) =1.
\end{equation*}%
Moreover, 
\begin{equation*}
\left\Vert u_{0}\right\Vert =\left\Vert x_{0}\right\Vert \text{ and }%
\left\Vert Q_{0}\right\Vert \geq \left\Vert Q_{0}\left( \frac{x_{0}}{%
\left\Vert x_{0}\right\Vert }\right) \right\Vert =\frac{1}{\left\Vert
x_{0}\right\Vert ^{m}}
\end{equation*}%
Hence,%
\begin{equation*}
\left\Vert Q_{0}\right\Vert _{\mathcal{P}_{L}^{+}}\left\Vert
u_{0}\right\Vert _{\mathcal{B}_{R}^{+}}^{m}\geq \left\Vert Q_{0}\right\Vert
\left\Vert u_{0}\right\Vert ^{m}\geq 1
\end{equation*}%
Taking the infimum over all possible factorizations of $P^{m}$, we conclude
that%
\begin{equation*}
\left\Vert P^{m}\right\Vert _{\mathcal{P}_{L}^{+}\circ \mathcal{B}%
_{R}^{+}}\geq 1.
\end{equation*}%
3) Let $\varphi \in F^{\ast }$ and $x\in E.$ Consider $u_{0}:\mathbb{K}%
\longrightarrow E$ defined by $u_{0}\left( \lambda \right) =\lambda x.$ We
have $\left\Vert u_{0}\right\Vert =\left\Vert x\right\Vert $ and%
\begin{equation*}
\varphi \circ P\circ u_{0}\left( \lambda \right) =\lambda ^{m}\left\langle
\varphi ,P\left( x\right) \right\rangle .
\end{equation*}%
Thus $\varphi \circ P\circ u_{0}=\left\langle \varphi ,P\left( x\right)
\right\rangle P^{m}.$ We have 
\begin{eqnarray*}
\left\vert \left\langle \varphi ,P\left( x\right) \right\rangle \right\vert
&=&\left\vert \left\langle \varphi ,P\left( x\right) \right\rangle
\right\vert \left\Vert P^{m}\right\Vert _{\mathcal{P}_{L}^{+}\circ \mathcal{B%
}_{R}^{+}}=\left\Vert \left\langle \varphi ,P\left( x\right) \right\rangle
P^{m}\right\Vert _{\mathcal{P}_{L}^{+}\circ \mathcal{B}_{R}^{+}} \\
&=&\left\Vert \varphi \circ P\circ u_{0}\right\Vert _{\mathcal{P}%
_{L}^{+}\circ \mathcal{B}_{R}^{+}} \\
&\leq &\left\Vert \varphi \right\Vert \left\Vert P\right\Vert _{\mathcal{P}%
_{L}^{+}\circ \mathcal{B}_{R}^{+}}\left\Vert u_{0}\right\Vert
^{m}=\left\Vert \varphi \right\Vert \left\Vert P\right\Vert _{\mathcal{P}%
_{L}^{+}\circ \mathcal{B}_{R}^{+}}\left\Vert x\right\Vert ^{m}.
\end{eqnarray*}%
Then%
\begin{eqnarray*}
\left\Vert P\left( x\right) \right\Vert &=&\sup_{\varphi \in B_{F^{\ast
}}}\left\vert \left\langle \varphi ,P\left( x\right) \right\rangle
\right\vert \leq \sup_{\varphi \in B_{F^{\ast }}}\left\Vert \varphi
\right\Vert \left\Vert P\right\Vert _{\mathcal{B}_{L}^{+}\circ \mathcal{B}%
_{R}^{+}}\left\Vert x\right\Vert ^{m} \\
&\leq &\left\Vert P\right\Vert _{\mathcal{B}_{L}^{+}\circ \mathcal{B}%
_{R}^{+}}\left\Vert x\right\Vert ^{m}.
\end{eqnarray*}%
Consequently, $\left\Vert P\right\Vert \leq \left\Vert P\right\Vert _{%
\mathcal{B}_{L}^{+}\circ \mathcal{B}_{R}^{+}}.$
\end{proof}

\begin{lemma}
The quantity $\left\Vert \cdot \right\Vert _{\mathcal{P}_{L}^{+}\circ 
\mathcal{B}_{R}^{+}}$ defined in $(\ref{2.1})$ can equivalently be expressed
as 
\begin{equation*}
\left\Vert P\right\Vert _{\mathcal{P}_{L}^{+}\circ \mathcal{B}_{R}^{+}}=\inf
\left\{ \left\Vert Q\right\Vert _{\mathcal{P}_{L}^{+}}:P=Q\circ u\text{ and }%
\left\Vert u\right\Vert _{\mathcal{B}_{R}^{+}}=1\right\} .
\end{equation*}
\end{lemma}

\begin{proof}
First, it is clear that 
\begin{equation*}
\left\Vert P\right\Vert _{\mathcal{P}_{L}^{+}\circ \mathcal{B}_{R}^{+}}\leq
\inf \left\{ \left\Vert Q\right\Vert _{\mathcal{P}_{L}^{+}}:P=Q\circ u\text{
and }\left\Vert u\right\Vert _{\mathcal{B}_{R}^{+}}=1\right\} .
\end{equation*}%
Consider a representation of $P$ as $Q_{0}\circ u_{0}.$ We can rewrite it as 
$P=(\left\Vert u_{0}\right\Vert _{\mathcal{B}_{R}^{+}}^{m}Q_{0})\circ (\frac{%
u_{0}}{\left\Vert u_{0}\right\Vert _{\mathcal{B}_{R}^{+}}}).$ Hence%
\begin{equation*}
\left\Vert \left\Vert u_{0}\right\Vert _{\mathcal{B}_{R}^{+}}^{m}Q_{0}\right%
\Vert _{\mathcal{P}_{L}^{+}}\geq \inf \left\{ \left\Vert Q\right\Vert _{%
\mathcal{P}_{L}^{+}}:P=Q\circ u\text{ and }\left\Vert u\right\Vert _{%
\mathcal{B}_{R}^{+}}=1\right\} .
\end{equation*}%
This implies%
\begin{equation*}
\left\Vert u_{0}\right\Vert _{\mathcal{B}_{R}^{+}}^{m}\left\Vert
Q_{0}\right\Vert _{\mathcal{P}_{L}^{+}}\geq \inf \left\{ \left\Vert
Q\right\Vert _{\mathcal{P}_{L}^{+}}:P=Q\circ u\text{ and }\left\Vert
u\right\Vert _{\mathcal{B}_{R}^{+}}=1\right\}
\end{equation*}%
Taking the infimum over all factorizations $P,$ we get%
\begin{equation*}
\left\Vert P\right\Vert _{\mathcal{P}_{L}^{+}\circ \mathcal{B}_{R}^{+}}\geq
\inf \left\{ \left\Vert Q\right\Vert _{\mathcal{P}_{L}^{+}}:P=Q\circ u\text{
and }\left\Vert u\right\Vert _{\mathcal{B}_{R}^{+}}=1\right\} .
\end{equation*}%
\bigskip
\end{proof}

The proof of the following theorem can be easily proved.

\begin{theorem}
\label{The1}If $\mathcal{B}_{R}^{+}$ is a positive right Banach ideal and $%
\mathcal{P}_{L}^{+}$ is a positive left Banach polynomial ideal, then%
\begin{equation*}
\left( \mathcal{P}_{L}^{+}\circ \mathcal{B}_{R}^{+},\left\Vert \cdot
\right\Vert _{\mathcal{P}_{L}^{+}\circ \mathcal{B}_{R}^{+}}\right) ,
\end{equation*}%
forms a positive quasi-Banach polynomial ideal.
\end{theorem}

\begin{proof}
We show that $\left\Vert .\right\Vert _{\mathcal{P}_{L}^{+}\circ \mathcal{B}%
_{R}^{+}}$ defines a quasi-norm, the remaining properties follow from
Proposition \ref{Propo1}. Let $\lambda \in \mathbb{K}$ and $P\in \mathcal{P}%
_{L}^{+}\circ \mathcal{B}_{R}^{+}(^{m}E;F).$ There exist a Banach space $X$
and elements $u_{0}\in \mathcal{B}_{R}^{+}\left( E;X\right) ,Q_{0}\in 
\mathcal{P}_{L}^{+}\left( ^{m}X;F\right) $ such that $P=Q_{0}\circ u_{0}.$
Then%
\begin{equation*}
\left\Vert \lambda P\right\Vert _{\mathcal{P}_{L}^{+}\circ \mathcal{B}%
_{R}^{+}}\leq \left\Vert \lambda Q_{0}\right\Vert _{\mathcal{P}%
_{L}^{+}}\left\Vert u_{0}\right\Vert _{\mathcal{B}_{R}^{+}}^{m}=\left\vert
\lambda \right\vert \left\Vert Q_{0}\right\Vert _{\mathcal{P}%
_{L}^{+}}\left\Vert u_{0}\right\Vert _{\mathcal{B}_{R}^{+}}^{m}.
\end{equation*}%
Taking the infimum over all factorizations of $P,$ we get 
\begin{equation*}
\left\Vert \lambda P\right\Vert _{\mathcal{P}_{L}^{+}\circ \mathcal{B}%
_{R}^{+}}\leq \left\vert \lambda \right\vert \left\Vert P\right\Vert _{%
\mathcal{P}_{L}^{+}\circ \mathcal{B}_{R}^{+}}.
\end{equation*}%
For the reverse inequality, assume $\lambda \neq 0.$ If $Q_{0}\circ u_{0}$
is a representation of $\lambda P,$ then $P=\frac{Q_{0}}{\lambda }\circ
u_{0},$ giving%
\begin{equation*}
\left\Vert P\right\Vert _{\mathcal{P}_{L}^{+}\circ \mathcal{B}_{R}^{+}}\leq
\left\Vert \frac{Q_{0}}{\lambda }\right\Vert _{\mathcal{P}%
_{L}^{+}}\left\Vert u_{0}\right\Vert _{\mathcal{B}_{R}^{+}}^{m}\leq \frac{1}{%
\left\vert \lambda \right\vert }\left\Vert Q_{0}\right\Vert _{\mathcal{P}%
_{L}^{+}}\left\Vert u_{0}\right\Vert _{\mathcal{B}_{R}^{+}}^{m}.
\end{equation*}%
Taking the infimum over all factorizations of $\lambda P,$ we obtain 
\begin{equation*}
\left\vert \lambda \right\vert \left\Vert P\right\Vert _{\mathcal{P}%
_{L}^{+}\circ \mathcal{B}_{R}^{+}}\leq \left\Vert \lambda P\right\Vert _{%
\mathcal{P}_{L}^{+}\circ \mathcal{B}_{R}^{+}}.
\end{equation*}%
By $(\ref{2.2})$ if $\left\Vert P\right\Vert _{\mathcal{P}_{L}^{+}\circ 
\mathcal{B}_{R}^{+}}=0,$ then $P=0.$ Let $P_{1},P_{2}\in \mathcal{P}%
_{L}^{+}\circ \mathcal{B}_{R}^{+}(^{m}E;F).$ Following a similar approach to
the proof of Proposition \ref{Propo}$,$ $P_{1}+P_{2}=B\circ A.$ We can then
establish the following inequalities%
\begin{eqnarray*}
\left\Vert A\right\Vert _{\mathcal{B}_{R}^{+}} &\leq &\left\Vert i_{1}\circ
u_{1}\right\Vert _{\mathcal{B}_{R}^{+}}+\left\Vert i_{2}\circ
u_{2}\right\Vert _{\mathcal{B}_{R}^{+}} \\
&\leq &\left\Vert i_{1}\right\Vert \left\Vert u_{1}\right\Vert _{\mathcal{B}%
_{R}^{+}}+\left\Vert i_{2}\right\Vert \left\Vert u_{2}\right\Vert _{\mathcal{%
B}_{R}^{+}}=\left\Vert u_{1}\right\Vert _{\mathcal{B}_{R}^{+}}+\left\Vert
u_{2}\right\Vert _{\mathcal{B}_{R}^{+}}.
\end{eqnarray*}%
Similarly,%
\begin{eqnarray*}
\left\Vert B\right\Vert _{\mathcal{P}_{L}^{+}} &\leq &\left\Vert Q_{1}\circ
\pi _{1}\right\Vert _{\mathcal{P}_{L}^{+}}+\left\Vert Q_{2}\circ \pi
_{2}\right\Vert _{\mathcal{P}_{L}^{+}} \\
&\leq &\left\Vert Q_{1}\right\Vert _{\mathcal{P}_{L}^{+}}\left\Vert \pi
_{1}\right\Vert ^{m}+\left\Vert Q_{2}\right\Vert _{\mathcal{P}%
_{L}^{+}}\left\Vert \pi _{2}\right\Vert ^{m}=\left\Vert Q_{1}\right\Vert _{%
\mathcal{P}_{L}^{+}}+\left\Vert Q_{2}\right\Vert _{\mathcal{P}_{L}^{+}}.
\end{eqnarray*}%
Now, for each $\varepsilon >0$ we can choose $u_{1},u_{2},Q_{1},Q_{2}$ such
that 
\begin{equation*}
\left\Vert Q_{j}\right\Vert _{\mathcal{P}_{L}^{+}}\leq \left\Vert
P_{j}\right\Vert _{\mathcal{P}_{L}^{+}\circ \mathcal{B}_{R}^{+}}+\varepsilon 
\text{ and }\left\Vert u_{j}\right\Vert _{\mathcal{B}_{R}^{+}}=1\text{ for }%
j=1,2.
\end{equation*}%
A simple calculation shows that%
\begin{eqnarray*}
\left\Vert P_{1}+P_{2}\right\Vert _{\mathcal{P}_{L}^{+}\circ \mathcal{B}%
_{R}^{+}} &\leq &\left\Vert B\right\Vert _{\mathcal{P}_{L}^{+}}\left\Vert
A\right\Vert _{\mathcal{B}_{R}^{+}} \\
&\leq &\left( \left\Vert u_{1}\right\Vert _{\mathcal{B}_{R}^{+}}+\left\Vert
u_{2}\right\Vert _{\mathcal{B}_{R}^{+}}\right) \left( \left\Vert
Q_{1}\right\Vert _{\mathcal{P}_{L}^{+}}+\left\Vert Q_{2}\right\Vert _{%
\mathcal{P}_{L}^{+}}\right) \\
&\leq &2\left( \left\Vert P_{1}\right\Vert _{\mathcal{P}_{L}^{+}\circ 
\mathcal{B}_{R}^{+}}+\left\Vert P_{2}\right\Vert _{\mathcal{P}_{L}^{+}\circ 
\mathcal{B}_{R}^{+}}+2\varepsilon \right)
\end{eqnarray*}%
Since $\varepsilon $ is arbitrary, it follows that%
\begin{equation*}
\left\Vert P_{1}+P_{2}\right\Vert _{\mathcal{P}_{L}^{+}\circ \mathcal{B}%
_{R}^{+}}\leq 2\left( \left\Vert P_{1}\right\Vert _{\mathcal{P}_{L}^{+}\circ 
\mathcal{B}_{R}^{+}}+\left\Vert P_{2}\right\Vert _{\mathcal{P}_{L}^{+}\circ 
\mathcal{B}_{R}^{+}}\right) .
\end{equation*}
\end{proof}

\textbf{The composition method. }Let $\mathcal{B}_{L}^{+}$ be a positive
left ideal. Let $X$ be a Banach space, and $E$ a Banach lattice. A
polynomial $P\in \mathcal{P}(^{m}X;E)$ belongs to $\mathcal{B}_{L}^{+}\circ 
\mathcal{P}$ if there exist a Banach space $Y$, a polynomial $Q\in \mathcal{P%
}(^{m}X;Y)$, and an operator $u\in \mathcal{B}_{L}^{+}(Y;E)$ such that 
\begin{equation*}
\begin{array}{ccc}
X & \overset{P}{\longrightarrow } & E \\ 
Q\downarrow & \nearrow u &  \\ 
Y &  & 
\end{array}%
\end{equation*}%
i.e., $P=u\circ Q$. In this case, we denote $P\in \mathcal{B}_{L}^{+}\circ 
\mathcal{P}(^{m}X;E).$

\begin{remark}
By an argument analogous to that used in \cite[Proposition 3.3]{BPR07}, the
class $\mathcal{B}_{L}^{+}\circ \mathcal{P}$ forms a positive left
polynomial ideal.
\end{remark}

If $\mathcal{B}_{L}^{+}$ is a positive left Banach ideal, we define%
\begin{equation}
\Vert P\Vert _{\mathcal{B}_{L}^{+}\circ \mathcal{P}}=\inf \{\Vert u\Vert _{%
\mathcal{B}_{L}^{+}}\Vert Q\Vert \},  \label{2.3}
\end{equation}%
where the infimum is taken over all possible factorizations of $P$ as
described above. Similarly to \cite[Proposition 3.7 ]{BPR07}, if $\mathcal{B}%
_{L}^{+}$ is a positive left Banach ideal, the pair $\left( \mathcal{B}%
_{L}^{+}\circ \mathcal{P},\Vert \cdot \Vert _{\mathcal{B}_{L}^{+}\circ 
\mathcal{P}}\right) $ forms a positive left Banach polynomial ideal. We have
the following result.

\begin{proposition}
\label{PropoCom}Let $\mathcal{B}_{L}^{+}$ be a positive left ideal. Let $X$
be a Banach space and $E$ a Banach lattice. For $P\in \mathcal{P}(^{m}X;E),$
the following statements are equivalent:\newline
1) The polynomial $P$ belongs to $\mathcal{B}_{L}^{+}\circ \mathcal{P}%
(^{m}X;E)$.\newline
2) The linearization $P_{L}$ belongs to $\mathcal{B}_{L}^{+}(\widehat{%
\otimes }_{\pi ,s}^{m}X;E).$\newline
Consequently, we obtain the following isometric identification%
\begin{equation*}
\mathcal{B}_{L}^{+}\circ \mathcal{P}(^{m}X;E)=\mathcal{B}_{L}^{+}(\widehat{%
\otimes }_{\pi ,s}^{m}X;E).
\end{equation*}
\end{proposition}

\begin{proof}
$1)\Longrightarrow 2):$ Let $P\in \mathcal{B}_{L}^{+}\circ \mathcal{P}%
(^{m}X;E).$ Then there exist $Q\in \mathcal{P}(^{m}X;Y)$ and $u\in \mathcal{B%
}_{L}^{+}(Y;E)$ such that $P=u\circ Q.$ Since $P_{L}=u\circ Q_{L},$ the
positive ideal property implies $P_{L}\in \mathcal{B}_{L}^{+}(\widehat{%
\otimes }_{\pi ,s}^{m}X;E).$ Moreover, as $P=P_{L}\circ \delta _{m}$ with $%
\delta _{m}\in \mathcal{P}\left( ^{m}X;\widehat{\otimes }_{\pi
,s}^{m}X\right) $ and $\left\Vert \delta _{m}\right\Vert =1.$ By $(\ref{2.3}%
) $ 
\begin{eqnarray*}
\left\Vert P_{L}\right\Vert _{\mathcal{B}_{L}^{+}} &=&\left\Vert u\circ
Q_{L}\right\Vert _{\mathcal{B}_{L}^{+}} \\
&\leq &\left\Vert u\right\Vert _{\mathcal{B}_{L}^{+}}\left\Vert
Q_{L}\right\Vert =\left\Vert u\right\Vert _{\mathcal{B}_{L}^{+}}\left\Vert
Q\right\Vert .
\end{eqnarray*}%
Taking the infimum over all such representations of $P$, it follows that%
\begin{equation*}
\left\Vert P_{L}\right\Vert _{\mathcal{B}_{L}^{+}}\leq \left\Vert
P\right\Vert _{\mathcal{B}_{L}^{+}\circ \mathcal{P}}.
\end{equation*}%
$2)\Longrightarrow 1):$ Suppose that $P_{L}\in \mathcal{B}_{L}^{+}(\widehat{%
\otimes }_{\pi ,s}^{m}X;E)$. Then 
\begin{equation*}
P=P_{L}\circ \delta _{m}\in \mathcal{B}_{L}^{+}\circ \mathcal{P}(^{m}X;E).
\end{equation*}%
Furthermore,%
\begin{equation*}
\left\Vert P_{L}\right\Vert _{\mathcal{B}_{L}^{+}}=\left\Vert
P_{L}\right\Vert _{\mathcal{B}_{L}^{+}}\left\Vert \delta _{m}\right\Vert
\geq \left\Vert P\right\Vert _{\mathcal{B}_{L}^{+}\circ \mathcal{P}}.
\end{equation*}%
For the surjectivity, let $R\in \mathcal{B}_{L}^{+}(\widehat{\otimes }_{\pi
,s}^{m}X;E).$ Define 
\begin{eqnarray*}
P_{R}\left( x\right) &=&R\left( x\otimes \overset{\left( m\right) }{...}%
\otimes x\right) \\
&=&R\circ \delta _{m}\left( x,...,x\right) .
\end{eqnarray*}%
We have $P_{R}\in \mathcal{P}(^{m}X;E)$ and 
\begin{equation*}
\widehat{R_{R}}=R\circ \delta _{m}.
\end{equation*}%
We verify that $\left( P_{R}\right) _{L}=R.$ Indeed, for any $%
\sum\limits_{i=1}^{n}\lambda _{i}x_{i}\otimes \overset{\left( m\right) }{...}%
\otimes x_{i}\in \otimes _{\pi ,s}^{m}X$,%
\begin{eqnarray*}
\left( P_{R}\right) _{L}(\sum\limits_{i=1}^{n}\lambda _{i}x_{i}\otimes 
\overset{\left( m\right) }{...}\otimes x_{i})
&=&\sum\limits_{i=1}^{n}\lambda _{i}\left( P_{R}\right) _{L}(x_{i}\otimes 
\overset{\left( m\right) }{...}\otimes x_{i}) \\
&=&\sum\limits_{i=1}^{n}\lambda _{i}P_{R}\left( x_{i}\right)
=\sum\limits_{i=1}^{n}\lambda _{i}R\left( x_{i}\otimes \overset{\left(
m\right) }{...}\otimes x_{i}\right) \\
&=&R(\sum\limits_{i=1}^{n}\lambda _{i}x_{i}\otimes \overset{\left( m\right) }%
{...}\otimes x_{i}).
\end{eqnarray*}%
Thus $\left( P_{R}\right) _{L}$ and $R$ coincide on $\otimes _{\pi ,s}^{m}X$%
, and by density they coincide on the whole space $\widehat{\otimes }_{\pi
,s}^{m}X.$
\end{proof}

\textbf{The factorization method. }Let $\mathcal{B}_{R}^{+}$ be a positive
right Banach ideal. We define the class $\mathcal{P}(\mathcal{B}_{R}^{+})$
as follows: Let $Y$ be a Banach space and $E$ a Banach lattice. A polynomial 
$P$ belongs to $\mathcal{P}(\mathcal{B}_{R}^{+})(^{m}E;Y)$ if there exist
Banach space $X,$ an operator $u\in \mathcal{B}_{R}^{+}(E;X),$ and a
polynomial $Q\in \mathcal{P}(^{m}X;Y)$ such that%
\begin{equation*}
\begin{array}{ccc}
E & \overset{P}{\longrightarrow } & Y \\ 
u\downarrow  & \nearrow Q &  \\ 
X &  & 
\end{array}%
\end{equation*}%
i.e., $P=Q\circ u$. In this case, for every $P\in \mathcal{P}(\mathcal{B}%
_{R}^{+})$ we define%
\begin{equation*}
\Vert P\Vert _{\mathcal{P}(\mathcal{B}_{R}^{+})}=\inf \{\Vert Q\Vert \Vert
u\Vert _{\mathcal{B}_{R}^{+}}^{m}\},
\end{equation*}%
where the infimum is taken over all possible factorizations of $P$ as
described above.

\begin{proposition}
Let $\mathcal{B}_{R}^{+}$ be a positive right Banach ideal. Then the pair $%
\left( \mathcal{P}(\mathcal{B}_{R}^{+}),\Vert \cdot \Vert _{\mathcal{P}(%
\mathcal{B}_{R}^{+})}\right) $ forms a positive right quasi-Banach
polynomial ideal.
\end{proposition}

\begin{proof}
We first verify the positive ideal property. Let $P\in \mathcal{P}(\mathcal{B%
}_{R}^{+})(^{m}E;Y),u\in \mathcal{L}^{+}\left( G;E\right) $ and $v\in 
\mathcal{L}\left( Y;Z\right) $. Suppose $P=Q_{0}\circ u_{0}$ is a
factorization of $P$. Then%
\begin{equation*}
v\circ P\circ u=v\circ \left( Q_{0}\circ u_{0}\right) \circ u
\end{equation*}%
Since $\left( u_{0}\circ u\right) \in \mathcal{B}_{R}^{+}\left( G;X\right) $
and $\left( v\circ Q_{0}\right) \in \mathcal{P}\left( ^{m}X;Z\right) $ we get%
\begin{equation*}
\left( v\circ Q_{0}\right) \circ \left( u_{0}\circ u\right) \in \mathcal{P}(%
\mathcal{B}_{R}^{+})\left( ^{m}G;Z\right) .
\end{equation*}%
The remaining steps follow exactly as in the proof of Proposition \ref{Propo}%
. For the quasi-norm, the argument is as in Theorem \ref{The1}, which shows
that $\Vert \cdot \Vert _{\mathcal{P}(\mathcal{B}_{R}^{+})}$ indeed defines
a quasi-norm.
\end{proof}

\section{\textsc{Examples of positive polynomial ideals}}

In this section, we give several classes of polynomials that serve as
concrete examples of positive polynomial ideals. These classes illustrate
the notions and properties discussed in the previous section. They include
polynomials that are Cohen positive strongly $p$-summing, Cohen positive $p$%
-nuclear, positive $p$-dominated and positive $\left( q,r\right) $%
-dominated. Studying these classes helps to understand how positivity
interacts with factorization, domination, and structural aspects in the
theory of polynomial ideals.

\subsection{\textsc{Cohen positive strongly }$p$\textsc{-summing polynomials}%
}

Hamdi et al. \cite{hamdi}, introduced the notion of Cohen positive strongly $%
p$-summing polynomials. A polynomial $P\in \mathcal{P}\left( ^{m}X;E\right) $
is said to be Cohen positive strongly $p$-summing if there exists a constant 
$C>0$ such that for any $(x_{i})_{i=1}^{n}\subset X$ and $\left( y_{i}^{\ast
}\right) _{i=1}^{n}\subset E^{\ast +}$, the following inequality holds: 
\begin{equation}
\sum_{i=1}^{n}\left\vert \left\langle P\left( x_{i}\right) ,y_{i}^{\ast
}\right\rangle \right\vert \leq C(\sum\limits_{i=1}^{n}\left\Vert
x_{i}\right\Vert ^{mp})^{\frac{1}{p}}\left\Vert \left( y_{i}^{\ast }\right)
_{i=1}^{n}\right\Vert _{p^{\ast },w}.  \label{3.1}
\end{equation}%
The space consisting of all such mappings is denoted by $\mathcal{P}%
_{Coh,p}^{+}\left( ^{m}X;E\right) $. In this case, we define 
\begin{equation*}
d_{p}^{m+}\left( P\right) =\inf \{C>0:C\text{ \textit{satisfies }}(\ref{3.1}%
)\}.
\end{equation*}

\begin{proposition}
The class $\mathcal{P}_{Coh,p}^{+}$ is a positive left Banach polynomial
ideal, obtained by the composition method from the positive left ideal $%
\mathcal{D}_{p}^{+}.$ More precisely, for every Banach space $X$ and Banach
lattice $E$%
\begin{equation*}
\mathcal{P}_{Coh,p}^{+}\left( ^{m}X;E\right) =\mathcal{D}_{p}^{+}\circ 
\mathcal{P}\left( ^{m}X;E\right) .
\end{equation*}
\end{proposition}

\begin{proof}
Directly by Proposition \ref{PropoCom} and \cite[Proposition 6]{hamdi}.
\end{proof}

\subsection{\textsc{Positive Cohen }$p$\textsc{-nuclear polynomials}}

Achour and Alouani \cite{AA10} introduced the notion of Cohen $p$-nuclear
multilinear operators as a natural extension of the linear concept
originally proposed by Cohen \cite{Coh73}. The polynomial counterpart was
later introduced and studied in \cite{AlouTh}. Hammou et al. \cite{Hammou}
subsequently developed the positive version of this notion.

\begin{definition}
Let $m\in \mathbb{N}$ and $1\leq p\leq \infty $. Let $E$ and $F$ be Banach
lattices. A polynomial $P\in \mathcal{P}\left( ^{m}E;F\right) $ is said to
be positive Cohen $p$-nuclear if there exists a constant $C>0$ such that for
any $(x_{i})_{i=1}^{n}\subset E^{+}$ and $\left( y_{i}^{\ast }\right)
_{i=1}^{n}\subset F^{\ast +}$, the following inequality holds: 
\begin{equation}
\sum_{i=1}^{n}\left\vert \left\langle P\left( x_{i}\right) ,y_{i}^{\ast
}\right\rangle \right\vert \leq C\sup_{x^{\ast }\in B_{E^{\ast }}^{+}}\left(
\sum\limits_{i=1}^{n}\left\langle x^{\ast },x_{i}\right\rangle ^{mp}\right)
^{\frac{1}{p}}\left\Vert \left( y_{i}^{\ast }\right) _{i=1}^{n}\right\Vert
_{p^{\ast },w}.  \label{3.2}
\end{equation}%
The space consisting of all such mappings is denoted by $\mathcal{P}_{N\text{%
-}p}^{c+}\left( ^{m}E;F\right) $. In this case, we define 
\begin{equation*}
n_{p}^{m+}\left( P\right) =\inf \{C>0:C\text{ \textit{satisfies }}(\ref{3.2}%
)\}.
\end{equation*}
\end{definition}

It is easy to verify that every Cohen $p$-nuclear is positive Cohen $p$%
-nuclear, i.e.,%
\begin{equation*}
\mathcal{P}_{p,N}^{c}\left( ^{m}E;F\right) \subset \mathcal{P}_{N\text{-}%
p}^{c+}\left( ^{m}E;F\right) .
\end{equation*}

\begin{proposition}
The class $\mathcal{P}_{N\text{-}p}^{c+}$ is a positive polynomial ideal
defined by, $\mathcal{P}_{N\text{-}p}^{c+}=\mathcal{P}_{Coh,p}^{+}\circ \Pi
_{p}^{+}.$ That is, for every pair of Banach lattices $E$ and $F$%
\begin{equation*}
\mathcal{P}_{N\text{-}p}^{c+}\left( ^{m}E;F\right) =\mathcal{P}%
_{Coh,p}^{+}\circ \Pi _{p}^{+}\left( ^{m}E;F\right) .
\end{equation*}
\end{proposition}

\begin{proof}
Directly by \cite[Theorem 9]{Hammou}.
\end{proof}

\subsection{\textsc{Positive }$p$\textsc{-dominated polynomials}}

The concept of positive $p$-dominated polynomials has been introduced by
Hamdi et al. \cite{hamdi}. A polynomial $P\in \mathcal{P}\left(
^{m}E;Y\right) $ is said to be positive $p$-dominated if there exists a
constant $C>0$ such that for any $(x_{i})_{i=1}^{n}\subset E^{+}$, the
following inequality holds: 
\begin{equation}
(\sum_{i=1}^{n}\left\Vert P\left( x_{i}\right) \right\Vert ^{\frac{p}{m}})^{%
\frac{m}{p}}\leq C\sup_{x^{\ast }\in B_{E^{\ast
}}^{+}}(\sum\limits_{i=1}^{n}\left\vert x^{\ast }\left( x_{i}\right)
\right\vert ^{p})^{\frac{m}{p}}.  \label{3.3}
\end{equation}%
The space consisting of all such mappings is denoted by $\mathcal{P}%
_{d,p}^{+}\left( ^{m}E;Y\right) $. In this case, we define 
\begin{equation*}
\delta _{p}^{+}\left( P\right) =\inf \{C>0:C\text{ \textit{satisfies }}(\ref%
{3.3})\}.
\end{equation*}%
We note that $p\geq m$, $\delta _{p}^{+}\left( \cdot \right) $ is a norm,
but for $p<m$, it is only a quasi-norm.

\begin{theorem}
\label{TheoD}\cite[Theorem 6]{hamdi} Let $1\leq p<\infty $; an $m$%
-homogeneous polynomial $P:E\rightarrow Y$ is positive $p$-dominated if
there are $C>0$ and a probability measure $\mu $ on $B_{E^{\ast }}^{+}$ such
that for every $x\in E^{+}$%
\begin{equation*}
\left\Vert P\left( x\right) \right\Vert \leq C\left( \int_{B_{E^{\ast
}}^{+}}\left\langle x^{\ast },x\right\rangle ^{p}d\mu \right) ^{\frac{m}{p}}.
\end{equation*}%
Moreover, the smallest $C$ is $\delta _{p}^{+}\left( P\right) .$
\end{theorem}

The authors in \cite{hamdi} did not provide a factorization result for this
class. In what follows, we present a version of Kwapie\'{n}'s theorem
concerning the class of positive $p$-dominated polynomials. This allows us
to establish that this class can be interpreted through the factorization
method from the positive class $\Pi _{p}^{+}.$

\begin{theorem}
Let $m\in \mathbb{N}$ and $1\leq p<\infty $. An $m$-homogeneous polynomial $%
P:E\rightarrow Y$ is positive $p$-dominated if and only if, there exist a
Banach space $X$, a positive $p$-summing operator $u:E\rightarrow X,$ and a
polynomial $Q\in \mathcal{P}\left( ^{m}X;Y\right) $ such that%
\begin{equation*}
P=Q\circ u.
\end{equation*}%
Moreover, 
\begin{equation*}
\delta _{p}^{+}(P)=\inf \left\{ \left\Vert Q\right\Vert \pi _{p}^{+}\left(
u\right) ^{m}:P=Q\circ u\right\} .
\end{equation*}
\end{theorem}

\begin{proof}
Let $P:E\rightarrow Y$ be an $m$-homogeneous polynomial such that $P=Q\circ
u $ where $u\in \Pi _{p}^{+}\left( E;X\right) $ and $Q\in \mathcal{P}\left(
^{m}X;Y\right) .$ Let $x\in E^{+}.$ We have%
\begin{eqnarray*}
\left\Vert P\left( x\right) \right\Vert &=&\left\Vert Q\circ u\left(
x\right) \right\Vert \\
&\leq &\left\Vert Q\right\Vert \left\Vert u\left( x\right) \right\Vert ^{m}.
\end{eqnarray*}%
Since $u$ is positive $p$-summing, by $(\ref{DomiSumming})$ we obtain%
\begin{equation*}
\left\Vert P\left( x\right) \right\Vert \leq \left\Vert Q\right\Vert \pi
_{p}^{+}\left( u\right) ^{m}\left( \int_{B_{E^{\ast }}^{+}}\left\langle
x^{\ast },x\right\rangle ^{p}d\mu \right) ^{\frac{m}{p}}.
\end{equation*}%
Then, $P$ is positive $p$-dominated and 
\begin{equation*}
\delta _{p}^{+}(P)\leq \left\Vert Q\right\Vert \pi _{p}^{+}\left( u\right)
^{m}.
\end{equation*}%
Taking the infimum over all representation of $P$, we get%
\begin{equation*}
\delta _{p}^{+}(P)\leq \inf \left\{ \left\Vert Q\right\Vert \pi
_{p}^{+}\left( u\right) ^{m}:P=Q\circ u\right\} .
\end{equation*}%
To prove the first implication. Let $P\in \mathcal{P}_{d,p}^{+}\left(
^{m}E;Y\right) .$ By Theorem \ref{TheoD}, there is a probability measure $%
\mu $ on $B_{E^{\ast }}^{+}$ such that for all $x\in E^{+}$ we have%
\begin{equation*}
\left\Vert P\left( x\right) \right\Vert \leq \delta _{p}^{+}\left( P\right)
\left( \int_{B_{E^{\ast }}^{+}}\left\langle x^{\ast },x\right\rangle
^{p}d\mu \right) ^{\frac{m}{p}}.
\end{equation*}%
We now consider the operator $u_{0}:E\rightarrow L_{p}\left( B_{E^{\ast
}}^{+},\mu \right) $ which is given by $u_{0}\left( x\right) (x^{\ast
})=x^{\ast }(x)$. Notice that for all $x\in E^{+},$ we have 
\begin{eqnarray*}
\left\Vert u_{0}\left( x\right) \right\Vert &=&\left( \int_{B_{E^{\ast
}}^{+}}\langle x,x^{\ast }\rangle ^{p}d\mu \right) ^{\frac{1}{p}} \\
&\leq &\left\Vert x\right\Vert .
\end{eqnarray*}%
Let $X=\overline{u_{0}(E)}^{L_{p}\left( B_{E^{\ast }}^{+},\mu \right) }$ be
the closure in $L_{p}\left( B_{E^{\ast }}^{+},\mu \right) $ of the range of $%
u_{0}$, and let $u:E\rightarrow X$ be the induced operator. Note that $u$ is
positive $p$-summing with $\pi _{p}^{+}\left( u\right) \leq 1$. Let $%
Q_{0}:u_{0}\left( E\right) \rightarrow Y$ be the polynomial operator defined
on $u_{0}\left( E\right) $ by 
\begin{equation*}
Q_{0}\left( u_{0}\left( x\right) \right) =P\left( x\right)
\end{equation*}%
this definition makes sense because%
\begin{equation*}
\left\Vert Q_{0}\left( u_{0}\left( x\right) \right) \right\Vert \leq
C\left\Vert u_{0}\left( x\right) \right\Vert ^{m}.
\end{equation*}%
It follows that $Q_{0}$ is continuous on $u_{0}\left( E\right) $ and has a
unique bounded polynomial extension $Q$ to $X$. Finally, $P=Q\circ u$ where $%
u\in \Pi _{p}^{+}\left( E;X\right) $ and $Q\in \mathcal{P}(^{m}X;Y)$ and
this ends the proof.
\end{proof}

\begin{corollary}
Let $E$ be a Banach lattice and $Y$ a Banach space. The class $\mathcal{P}%
_{d,p}^{+}$ is a positive right polynomial ideal obtained through the
factorization method from the positive right ideal $\Pi _{p}^{+}.$
Specifically, for every Banach lattice $E$ and Banach space $Y,$ 
\begin{equation*}
\mathcal{P}_{d,p}^{+}\left( ^{m}E;Y\right) =\mathcal{P}\left( \Pi
_{p}^{+}\right) (^{m}E;Y).
\end{equation*}
\end{corollary}

\subsection{\textsc{Positive }$(q;r)$\textsc{-dominated polynomials}}

The concept of absolutely $(p,q,r)$-summing operators was first introduced
by Pietsch \cite{PIETSCHoi}. It was later extended to the multilinear
setting by Achour \cite{Ach11}, and to the polynomial setting by Achour and
Bernardino \cite{AchB}. A positive counterpart of this notion was
subsequently introduced and investigated in \cite{FerSaaP}. In this section,
we develop and analyze the corresponding positive polynomial version, which
serves as a natural example of a positive polynomial ideal.

\begin{definition}
\label{definition7}Let $m\in \mathbb{N}.$ Let $1\leq r,p,q\leq \infty $ with 
$\frac{1}{p}=\frac{m}{q}+\frac{1}{r}$. Let $E$ and $F$ be Banach lattices. A
polynomial $P\in \mathcal{P}\left( ^{m}E;F\right) $ is called positive $(q;r)
$-dominated if there exists a constant $C>0$ such that for any $%
(x_{i})_{i=1}^{n}\subset E^{+}$ and $\left( y_{i}^{\ast }\right)
_{i=1}^{n}\subset F^{\ast +}$, the following inequality holds: 
\begin{equation}
\left\Vert \left( \left\langle P\left( x_{i}\right) ,y_{i}^{\ast
}\right\rangle \right) _{i=1}^{n}\right\Vert _{p}\leq C\left\Vert \left(
x_{i}\right) _{i=1}^{n}\right\Vert _{q,w}^{m}\left\Vert \left( y_{i}^{\ast
}\right) _{i=1}^{n}\right\Vert _{r,w}.  \label{def1sec3}
\end{equation}%
The space of all such polynomials is denoted by $\mathcal{P}_{d,\left(
q;r\right) }^{+}\left( ^{m}E;F\right) $. Its norm is given by 
\begin{equation*}
d_{d,\left( q;r\right) }^{+}(P)=\inf \{C>0:C\text{ \textit{satisfies }}(\ref%
{def1sec3})\}.
\end{equation*}
\end{definition}

An equivalent formulation of $(\ref{def1sec3})$ is 
\begin{equation*}
\left\Vert \left( \left\langle P\left( x_{i}\right) ,y_{i}^{\ast
}\right\rangle \right) _{i=1}^{n}\right\Vert _{p}\leq C\left\Vert \left(
\left\vert x_{i}\right\vert \right) _{i=1}^{n}\right\Vert
_{q,w}^{m}\left\Vert \left( \left\vert y_{i}^{\ast }\right\vert \right)
_{i=1}^{n}\right\Vert _{r,w}
\end{equation*}%
for every $(x_{i})_{i=1}^{n}\subset E$ and $\left( y_{i}^{\ast }\right)
_{i=1}^{n}\subset F^{\ast }.$ It is straightforward to check that every $%
(q;r)$-dominated polynomial is positive $(q;r)$-dominated. Hence, by \cite[%
Proposition 3.10]{AchB} 
\begin{equation}
\mathcal{P}_{f}(^{m}E;F)\subset \mathcal{P}_{d,\left( q;r\right) }^{+}\left(
^{m}E;F\right) .  \label{4.1}
\end{equation}

\begin{proposition}
Let $P\in \mathcal{P}_{d,\left( p;r\right) }^{+}\left( ^{m}E;F\right) ,u\in 
\mathcal{L}^{+}\left( G;E\right) $ and $v\in \mathcal{L}^{+}(F;H).$ Then $%
v\circ P\circ u\in \mathcal{P}_{d,\left( q;r\right) }^{+}\left(
^{m}G;H\right) $ and we have%
\begin{equation*}
d_{d,\left( q;r\right) }^{+}\left( v\circ P\circ u\right) \leq \left\Vert
v\right\Vert d_{d\left( q;r\right) }^{+}(P)\left\Vert u\right\Vert ^{m}.
\end{equation*}
\end{proposition}

\begin{proof}
Let $(x_{i})_{i=1}^{n}\subset E^{+}$ and $\left( y_{i}^{\ast }\right)
_{i=1}^{n}\subset F^{\ast +}.$ Then%
\begin{eqnarray*}
(\sum\limits_{i=1}^{n}\left\vert \left\langle v\circ P\circ u\left(
x_{i}\right) ,y_{i}^{\ast }\right\rangle \right\vert ^{p})^{\frac{1}{p}}
&=&(\sum\limits_{i=1}^{n}\left\vert \left\langle P\circ u\left( x_{i}\right)
,v^{\ast }\circ y_{i}^{\ast }\right\rangle \right\vert ^{p})^{\frac{1}{p}} \\
&\leq &d_{d\left( q;r\right) }^{+}(P)\left\Vert \left( u\left( x_{i}\right)
\right) _{i=1}^{n}\right\Vert _{q,w}^{m}\left\Vert \left( v^{\ast }\circ
y_{i}^{\ast }\right) _{i=1}^{n}\right\Vert _{r,w} \\
&\leq &d_{d\left( q;r\right) }^{+}(P)\left\Vert u\right\Vert ^{m}\left\Vert
\left( x_{i}\right) _{i=1}^{n}\right\Vert _{q,w}^{m}\left\Vert v^{\ast
}\right\Vert \left\Vert \left( y_{i}^{\ast }\right) _{i=1}^{n}\right\Vert
_{r,w} \\
&\leq &\left\Vert v\right\Vert d_{d\left( q;r\right) }^{+}(P)\left\Vert
u\right\Vert ^{m}\left\Vert \left( x_{i}\right) _{i=1}^{n}\right\Vert
_{q,w}^{m}\left\Vert \left( y_{i}^{\ast }\right) _{i=1}^{n}\right\Vert _{r,w}
\end{eqnarray*}%
thus $v\circ P\circ u$ is positive $(q;r)$-dominated and%
\begin{equation*}
d_{d,\left( q;r\right) }^{+}\left( v\circ P\circ u\right) \leq \left\Vert
v\right\Vert d_{d\left( q;r\right) }^{+}(P)\left\Vert u\right\Vert ^{m}.
\end{equation*}
\end{proof}

The pair $\left( \mathcal{P}_{d,\left( q;r\right) }^{+},d_{d,\left(
q;r\right) }^{+}\right) $ defines a positive Banach polynomial ideal. The
proof follows directly from the previous Proposition and the inclusion $(\ref%
{4.1})$, while the remaining details are straightforward. We now turn to the
characterization of positive $(q;r)$-dominated polynomials through a
Pietsch-type domination theorem. To this end, we apply the general Pietsch
domination theorem established by Pellegrino et al in \cite[Theorem 4.6]%
{PSS12}.

\begin{theorem}[Pietsch domination theorem]
\label{thdo1}Let $m\in \mathbb{N}.$ Let $1\leq r,p,q\leq \infty $ with $%
\frac{1}{p}=\frac{m}{q}+\frac{1}{r}$. Let $E$ and $F$ be Banach lattices.
The following statements are equivalent:

1) The polynomial $P\in \mathcal{P}\left( ^{m}E;F\right) $ is positive $(q;r)
$-dominated.

2) There is a constant $C>0$ and Borel probability measures $\mu $ on $%
B_{E^{\ast }}^{+}$ and $\eta $ on $B_{F^{\ast \ast }}^{+}$ such that%
\begin{equation}
\left\vert \langle P(x),y^{\ast }\rangle \right\vert \leq C(\int_{B_{E^{\ast
}}^{+}}\langle \left\vert x\right\vert ,x^{\ast }\rangle ^{q}d\mu )^{\frac{m%
}{q}}(\int_{B_{F^{\ast \ast }}^{+}}\langle \left\vert y^{\ast }\right\vert
,y^{\ast \ast }\rangle ^{r}d\eta )^{\frac{1}{r}}  \label{1234}
\end{equation}%
for all $(x,y^{\ast })\in E\times F^{\ast }$. Therefore, we have%
\begin{equation*}
d_{d,\left( q;r\right) }^{+}(P)=\inf \{C>0:C\text{ \textit{satisfies }}(\ref%
{1234})\}.
\end{equation*}%
3) There is a constant $C>0$ and Borel probability measures $\mu $ on $%
B_{E^{\ast }}^{+}$ and $\eta $ on $B_{F^{\ast \ast }}^{+}$ such that%
\begin{equation}
\left\vert \langle P(x),y^{\ast }\rangle \right\vert \leq C(\int_{B_{E^{\ast
}}^{+}}\langle x,x^{\ast }\rangle ^{q}d\mu )^{\frac{m}{q}}(\int_{B_{F^{\ast
\ast }}^{+}}\langle y^{\ast },y^{\ast \ast }\rangle ^{r}d\eta )^{\frac{1}{r}}
\label{def2sec5}
\end{equation}%
for all $(x,y^{\ast })\in E^{+}\times F^{\ast +}$. Therefore, we have%
\begin{equation*}
d_{d,\left( q;r\right) }^{+}(P)=\inf \{C>0:C\text{ \textit{satisfies }}(\ref%
{def2sec5})\}.
\end{equation*}
\end{theorem}

\begin{proof}
$1)\Leftrightarrow 2):$ We will choose the parameters as specified in \cite[%
Theorem 4.6]{PSS12} 
\begin{equation*}
\left\{ 
\begin{array}{l}
S:\mathcal{P}\left( ^{m}E;F\right) \times \left( E\times F^{\ast }\right)
\times \mathbb{K\times K}\rightarrow \mathbb{R}^{+}: \\ 
S\left( P,\left( x,y^{\ast }\right) ,\lambda _{1},\lambda _{2}\right)
=\left\vert \lambda _{2}\right\vert |\langle P(x),y^{\ast }\rangle | \\ 
R_{1}:B_{E^{\ast }}^{+}\times \left( E\times F^{\ast }\right) \times \mathbb{%
K}\rightarrow \mathbb{R}^{+}:R_{1}(x^{\ast },\left( x,y^{\ast }\right)
,\lambda _{1})=\langle |x|,x^{\ast }\rangle ^{m} \\ 
R_{2}:B_{F^{\ast \ast }}^{+}\times \left( E\times F^{\ast }\right) \times 
\mathbb{K}\rightarrow \mathbb{R}^{+}:R_{2}(y^{\ast \ast },\left( x,y^{\ast
}\right) ,\lambda _{2})=\left\vert \lambda _{2}\right\vert \langle |y^{\ast
}|,y^{\ast \ast }\rangle .%
\end{array}%
\right.
\end{equation*}%
These maps satisfy conditions $\left( 1\right) $ and $\left( 2\right) $ from 
\cite[Theorem 4.6]{PSS12}, allowing us to conclude that $T:X\times
E\rightarrow F$ is dominated $(p,q)$-summing if and only if. We can easily
conclude that $P:E\rightarrow F$ is positive $(q;r)$-dominated if, and only
if, 
\begin{eqnarray*}
&&(\sum\limits_{i=1}^{n}S\left( P,\left( x_{i},y_{i}^{\ast }\right) ,\lambda
_{i,1},\lambda _{i,2}\right) ^{p})^{\frac{1}{p}} \\
&\leq &C\sup_{x^{\ast }\in B_{E^{\ast
}}^{+}}(\sum\limits_{i=1}^{n}R_{1}(x^{\ast },\left( x_{i},y_{i}^{\ast
}\right) ,\lambda _{i,1})^{q})^{\frac{m}{q}}\sup_{y^{\ast \ast }\in
B_{F^{\ast \ast }}^{+}}(\sum\limits_{i=1}^{n}R_{2}(y^{\ast \ast },\left(
x_{i},y_{i}^{\ast }\right) ,\lambda _{i,2})^{r})^{\frac{1}{r}},
\end{eqnarray*}%
i.e., $P$ is $R_{1},R_{2}$-$S$-abstract $(\frac{q}{m};r)$-summing. As
outlined in \cite[Theorem 4.6]{PSS12}, this implies that $P$ is $R_{1},R_{2}$%
-$S$-abstract $(\frac{q}{m};r)$-summing if, and only if, there exists a
positive constant $C$ and probability measures $\mu $ on $B_{E^{\ast }}^{+}$
and $\eta $ on $B_{F^{\ast \ast }}^{+}$, such that 
\begin{eqnarray*}
&&S\left( P,\left( x,y^{\ast }\right) ,\lambda _{1},\lambda _{2}\right) \\
&\leq &C(\int_{B_{E^{\ast }}^{+}}R_{1}(x^{\ast },\left( x,y^{\ast }\right)
,\lambda _{1})^{q}d\mu )^{\frac{m}{q}}(\int_{B_{F^{\ast \ast
}}^{+}}R_{2}(y^{\ast \ast },\left( x,y^{\ast }\right) ,\lambda
_{2})^{r}d\eta )^{\frac{1}{r}}.
\end{eqnarray*}%
Consequently 
\begin{equation*}
\left\vert \langle P(x),y^{\ast }\rangle \right\vert \leq C(\int_{B_{E^{\ast
}}^{+}}\langle \left\vert x\right\vert ,x^{\ast }\rangle ^{q}d\mu )^{\frac{m%
}{q}}(\int_{B_{F^{\ast \ast }}^{+}}\langle \left\vert y^{\ast }\right\vert
,y^{\ast \ast }\rangle ^{r}d\eta )^{\frac{1}{r}},
\end{equation*}%
The implications $2)\Longrightarrow 3)$ and $3)\Longrightarrow 1)$ are
immediate.
\end{proof}

As an immediate consequence of Theorem \ref{thdo1}, we can show that if $%
q_{1}\leq q_{2}$ and $r\leq s$ then 
\begin{equation*}
\mathcal{P}_{d,\left( q_{1};r\right) }^{+}\left( ^{m}E;F\right) \subset 
\mathcal{P}_{d,\left( q_{2};s\right) }^{+}\left( ^{m}E;F\right) .
\end{equation*}

The following result shows that the class of positive $\left( q;r\right) $%
-dominated polynomials can be represented as the composition of the class of
Cohen positive strongly $r^{\ast }$-summing polynomials $\mathcal{P}%
_{r^{\ast }}^{+}$ with the class of positive $p$-summing operators $\Pi
_{p}^{+}$. This provides a positive analogue of the Kwapie\'{n}
factorization.

\begin{theorem}
\label{thfa}Let $m\in \mathbb{N}$. Let $1\leq r,p,q\leq \infty $ with $\frac{%
1}{p}=\frac{m}{q}+\frac{1}{r}$. Then, $P\in \mathcal{P}(^{m}E;F)$ is
positive $\left( q;r\right) $-dominated if and only if there exist Banach
space $X$, a Cohen positive strongly $r^{\ast }$-summing polynomial $%
Q:X\rightarrow F$ and a positive $q$-summing operator $u\in \Pi
_{q}^{+}\left( E;X\right) $ so that $T=Q\circ u$, i.e., 
\begin{equation*}
\mathcal{P}_{d,\left( q;r\right) }^{+}(^{m}E;F)=\mathcal{P}_{Coh,r^{\ast
}}^{+}\circ \Pi _{q}^{+}(^{m}E;F).
\end{equation*}%
Moreover, 
\begin{equation*}
d_{d,\left( q;r\right) }^{+}(P)=\inf \left\{ d_{r^{\ast }}^{m+}(Q)\pi
_{q}^{+}\left( u\right) ^{m}:P=Q\circ u\right\} .
\end{equation*}
\end{theorem}

\begin{proof}
First we prove the converse. Suppose that $P=Q\circ u$ where $u$ is positive 
$q$-summing and $Q$ is Cohen positive strongly $r^{\ast }$-summing
polynomial. By \cite[Theorem 2.5]{BB18}, there exists $\eta $ on $B_{F^{\ast
\ast }}^{+}$ such that, for all $x\in E^{+}$ and $y^{\ast }\in F^{\ast +},\ $%
we have 
\begin{eqnarray*}
\left\vert \left\langle P(x),y^{\ast }\right\rangle \right\vert
&=&\left\vert \left\langle Q\left( u\left( x\right) \right) ,y^{\ast
}\right\rangle \right\vert \\
&\leq &d_{r^{\ast }}^{m}(Q)\left\Vert u\left( x\right) \right\Vert
^{m}\left( \int_{B_{F^{\ast \ast }}^{+}}\left\langle y^{\ast },y^{\ast \ast
}\right\rangle ^{r}d\eta \right) ^{\frac{1}{r}}.
\end{eqnarray*}%
Since $u$ is positive $q$-summing then\textbf{, }by (\ref{DomiSumming})
there is a probability measure $\mu $ on $B_{E^{\ast }}^{+}$ such that 
\begin{equation*}
\left\Vert u\left( x\right) \right\Vert \leq \pi _{p}^{+}\left( u\right)
\left( \int_{B_{E^{\ast }}^{+}}\langle x,x^{\ast }\rangle ^{q}d\mu \right) ^{%
\frac{1}{q}}.
\end{equation*}%
Consequently,%
\begin{equation*}
\left\vert \left\langle P(x),y^{\ast }\right\rangle \right\vert \leq
d_{r^{\ast }}^{m}(Q)\pi _{p}^{+}\left( u\right) \left( \int_{B_{E^{\ast
}}^{+}}\left\langle x,x^{\ast }\right\rangle ^{q}d\mu \right) ^{\frac{m}{q}%
}\left( \int_{B_{F^{\ast \ast }}^{+}}\left\langle y^{\ast },y^{\ast \ast
}\right\rangle ^{r}d\eta \right) ^{\frac{1}{r}}.
\end{equation*}%
Thus, $P$ is positive $(q;r)$-dominated by Theorem \ref{thdo1} and 
\begin{equation*}
d_{d,\left( q;r\right) }^{+}(P)\leq d_{r^{\ast }}^{+}(Q)\pi _{p}^{+}\left(
u\right) ^{m}.
\end{equation*}%
Taking the infimum over all representations $P$, we get%
\begin{equation*}
d_{d,\left( q;r\right) }^{+}(P)\leq \inf \left\{ d_{r^{\ast }}^{+}(Q)\pi
_{q}^{+}\left( u\right) ^{m}:P=Q\circ u\right\} .
\end{equation*}%
We now prove the direct implication. Let $P\in \mathcal{P}_{d,\left(
q;r\right) }^{+}(^{m}E;F).$ By Theorem \ref{thdo1}, there are probability
measures $\mu $ on $B_{E^{\ast }}^{+}$ and $\eta $ on $B_{F^{\ast \ast
}}^{+} $ such that for all $x\in E^{+}$ and $y^{\ast }\in F^{\ast +}$ we have%
\begin{equation*}
\left\vert \left\langle P(x),y^{\ast }\right\rangle \right\vert \leq
d_{d,\left( q;r\right) }^{+}(P)\left( \int_{B_{E^{\ast }}^{+}}\left\langle
x,x^{\ast }\right\rangle ^{q}d\mu \right) ^{\frac{m}{q}}\left(
\int_{B_{F^{\ast \ast }}^{+}}\left\langle y^{\ast },y^{\ast \ast
}\right\rangle ^{r}d\eta \right) ^{\frac{1}{r}}.
\end{equation*}%
Define the operator $u_{0}:E\rightarrow L_{q}\left( B_{E^{\ast }}^{+},\mu
\right) $ by $u_{0}\left( x\right) (x^{\ast })=x^{\ast }(x)$. For every $%
x\in E^{+},$ we have 
\begin{equation*}
\left\Vert u_{0}\left( x\right) \right\Vert =\left( \int_{B_{E^{\ast
}}^{+}}\left\langle x,x^{\ast }\right\rangle ^{q}d\mu \right) ^{\frac{1}{q}%
}\leq \left\Vert x\right\Vert .
\end{equation*}%
Let $X=\overline{u_{0}(E)}^{L_{q}\left( B_{E^{\ast }}^{+},\mu \right) }$,
and denote by $\overline{u_{0}}:E\rightarrow X$ the induced operator. Then, $%
\overline{u_{0}}$ is positive $q$-summing with $\pi _{q}^{+}\left( \overline{%
u_{0}}\right) \leq 1$. Now define the polynomial operator $Q_{0}$ on $%
u_{0}\left( E\right) $ by 
\begin{equation*}
Q_{0}\left( u_{0}\left( x\right) \right) =P\left( x\right) .
\end{equation*}%
This definition is consistent because 
\begin{equation*}
\left\vert \left\langle Q_{0}\left( u_{0}\left( x\right) \right) ,y^{\ast
}\right\rangle \right\vert \leq d_{d,\left( q;r\right) }^{+}(P)\left\Vert
u_{0}\left( x\right) \right\Vert ^{m}\left( \int_{B_{F^{\ast \ast
}}^{+}}\left\langle y^{\ast },y^{\ast \ast }\right\rangle ^{r}d\eta \right)
^{\frac{1}{r}}
\end{equation*}%
Hence $Q_{0}$ is continuous on $u_{0}\left( E\right) ,$ and extends uniquely
to a bounded polynomial $\overline{Q_{0}}$ on $X$. Moreover, $\overline{Q_{0}%
}$ is Cohen positive strongly $r^{\ast }$-summing polynomial and%
\begin{equation*}
d_{r^{\ast }}^{+}(\overline{Q_{0}})\leq d_{d,\left( q;r\right) }^{+}(P).
\end{equation*}%
Finally, we obtain $P=\overline{Q_{0}}\circ \overline{u_{0}}$ where $%
\overline{u_{0}}\in \Pi _{q}^{+}\left( E;X\right) ,\overline{Q_{0}}\in 
\mathcal{P}_{r^{\ast }}^{+}(^{m}X;F)$ and%
\begin{eqnarray*}
\inf \left\{ d_{r^{\ast }}^{+}(Q)\pi _{q}^{+}\left( u\right) ^{m}:P=Q\circ
u\right\} &\leq &d_{r^{\ast }}^{+}(\overline{Q_{0}})\pi _{q}^{+}\left( 
\overline{u_{0}}\right) ^{m} \\
&\leq &d_{d,\left( q;r\right) }^{+}(P).
\end{eqnarray*}%
Which completes the proof.
\end{proof}

Every positive $\left( q;r\right) $-dominated polynomial can be factored
through a Cohen positive strongly $r^{\ast }$-summing polynomial and
positive $q$-summing linear operator. Consequently, the class $\mathcal{P}%
_{d,\left( q;r\right) }^{+}$ forms a positive polynomial ideal of type $%
\mathcal{P}_{L}^{+}\circ \mathcal{B}_{R}^{+}$ where%
\begin{equation*}
\mathcal{P}_{L}^{+}=\mathcal{P}_{Coh,r^{\ast }}^{+}\text{ and }\mathcal{B}%
_{R}^{+}=\Pi _{q}^{+}.
\end{equation*}

\textbf{Declarations}\newline

\textbf{Conflict of interest.} The authors declare that they have no
conflicts of interest.\newline


\begin{thebibliography}{99}
\bibitem{Ach11} \textsc{D. Achour, }\textit{Multilinear extensions of
absolutely }$(p;q;r)$\textit{-summing operators}{.} Rend. Circ. Mat. Palermo
60, 337-350 (2011).

\bibitem{AA10} \textsc{D. Achour, A. Alouani, }\textit{On multilinear
generalizations of the concept of nuclear operators}. Colloq. Math 120(1),
85-102 (2010).

\bibitem{AB14} \textsc{D. Achour and A. Belacel, }\textit{Domination and
factorization theorems for positive strongly }$p$\textit{-summing operators}%
. Positivity 18, 785-804 (2014).

\bibitem{AchB} \textsc{D. Achour and A.T. Bernardino,} $(q;r)$-\textit{%
Dominated holomorphic mappings, }Collect. Math. DOI
10.1007/s13348-012-0073-0, (2012).

\bibitem{Bla87} \textsc{O. Blasco, }\textit{Positive }$p$\textit{-summing
operators on }$L_{p}$\textit{-spaces}. Proceedings of the American
Mathematical Society 100.2, 275-280 (1987).

\bibitem{AlouTh} \textsc{A. Alouani, }Sur les id\'{e}aux d'op\'{e}rateurs
sommants. Doctoral Dissertation, University of M'sila (2018).

\bibitem{BB18} \textsc{A. Bougoutaia and A. Belacel, }\textit{Cohen positive
strongly }$p$\textit{-summing and }$p$\textit{-convex multilinear operators}%
. Positivity 23.2, 379-395 (2019).

\bibitem{BBH21} \textsc{A. Bougoutaia, A. Belacel and H. Hamdi, }\textit{%
Domination and Kwapi\'{e}n factorization theorems for positive Cohen nuclear
linear operators}. Moroccan Journal of Pure and Applied Analysis 7.1,
100-115 (2021).

\bibitem{BBR23} \textsc{A. Bougoutaia, A. Belacel and P. Rueda, }\textit{%
Summability of multilinear operators and their linearizations on Banach
lattices}. J. Math. Anal. Appl. 527, 127459, (2023).

\bibitem{BPR07} \textsc{G. Botelho, D. Pellegrino and P. Rueda, }\textit{On
composition ideals of multilinear mappings and homogeneous polynomials}.
Publications of the Research Institute for Mathematical Sciences 43.4,
1139-1155, (2007).

\bibitem{CBD21} \textsc{D. Chen, A. Belacel, J. A. Ch\'{a}vez-Dom\'{\i}%
nguez, }\textit{Positive }$p$\textit{-summing operators and disjoint }$p$%
\textit{-summing operators}. Positivity 25, 1045-1077, (2021).

\bibitem{Coh73} \textsc{J.S. Cohen, }\textit{Absolutely }$p$\textit{%
-summing, }$p$\textit{-nuclear operators and their conjugates}{.} Math. Ann
201, 177-200 (1973).

\bibitem{defl} \textsc{A. Defant and K. Floret, }Tensor Norms and Operator
Ideals. North-Holland Publishing, North-Holland, 1993.

\bibitem{Dini} \textsc{S. Dineen, }Complex Analysis on Infinite Dimensional
Spaces, Springer Monographs in Mathematics, Springer-Verlag, London, London,
1999

\bibitem{FerSaaP} \textsc{A. Ferradi, A. Belaada and K. Saadi, }\textit{%
Positive ideals of multilinear operators}. Positivity 29, 20 (2025).
https://doi.org/10.1007/s11117-025-01112-4.

\bibitem{Flo} \textsc{K. Floret, }Natural norms on symmetric tensor-products
of normed spaces, Note Mat. 17 (1997), 153--188.

\bibitem{hamdi} \textsc{H. Hamdi, S. Garc\'{\i}a-Hern\'{a}ndez, A. Belacel,
D. Ouchenane,and S. Ahmed Zubair, }\textit{Cohen Positive Strongly p-Summing
and p-Dominated m-Homogeneous Polynomials. }Hindawi Journal of Mathematics,
V. 2021, Article ID 5495605, 18 pages (2021).

\bibitem{Hammou} \textsc{A. Hammou, A. Belacel, A. Bougoutaia and A. Tiaiba, 
}\textit{Positive Cohen }$p$\textit{-nuclear }$m$\textit{-homogeneous
polynomials}. Surveys in Mathematics and its Applications, V. 18, 107-121
(2023).

\bibitem{MN91} \textsc{P. Meyer-Nieberg, }Banach lattices. Springer, Berlin,
1991.

\bibitem{Muji} \textsc{J. Mujica, }Complex Analysis in Banach Spaces,
Holomorphic Functions and Domains of holomorphy in Finite and Infinite
Dimensions, North-Holland Mathematics Studies 120, Amsterdam, 1986.

\bibitem{PSS12} \textsc{D. Pellegrino, J. Santos\ and J.B.S. Sep\'{u}lveda, }%
\textit{Some techniques on nonlinear analysis and applications}{.} Adv. Math
229, 1235-1265 (2012).

\bibitem{PIETSCHoi} \textsc{A. Pietsch, }Operator Ideals. North-Holland
Publications, North-Holland, 1980.

\bibitem{Zhu98} \textsc{O.I. Zhukova, }\textit{On modifications of the
classes of }$p$\textit{-nuclear, }$p$\textit{-summing, and }$p$\textit{%
-integral operators}. Sib Math J 39, 894-907 (1998).
\end{thebibliography}
\end{document}